\documentclass[11pt]{amsart}
\usepackage{amsmath,a4wide,scalerel}
\usepackage{mathabx}
\usepackage{xcolor,graphicx}
\usepackage{mathrsfs}
\usepackage{pgfplots}
\usepackage{subcaption}
\usepackage[normalem]{ulem}
\usepackage{float}
\usepackage{ctable} % for \specialrule command

\usepackage{hyperref}
\newtheorem{theorem}{Theorem}
\newtheorem{lemma}[theorem]{Lemma}
\newtheorem{proposition}[theorem]{Proposition}

\newtheorem{corollary}[theorem]{Corollary}

\newtheorem{assumption}{Assumption}

\newcommand{\R}{{\mathbb R}}
\newcommand{\N}{{\mathbb N}}
\newcommand{\E}{{\mathbb E}}

\definecolor{dkgreen}{rgb}{0,0.6,0}

\newcommand{\Lg}{{\rm L_g}} %notation for Lipschitz constant of g

\newcommand{\LT}{{\rm LT}}%notation for the scheme

\usepackage{color}
\usepackage{dsfont}

\title[]{Analysis of a positivity-preserving splitting scheme for some nonlinear stochastic heat equations}

\date{\today}

\author{Charles-Edouard Br\'ehier}
			\address{Universit\'e de Pau et des Pays de l'Adour, E2S UPPA, CNRS, LMAP, Pau, France}
			\email{charles-edouard.brehier@univ-pau.fr}

\author{David Cohen}
              \address{Department of Mathematical Sciences,
              Chalmers University of Technology and University of Gothenburg, 41296~Gothenburg, Sweden}
              \email{\tt david.cohen@chalmers.se}

\author{Johan Ulander}
              \address{Department of Mathematical Sciences,
              Chalmers University of Technology and University of Gothenburg, 41296~Gothenburg, Sweden}
              \email{\tt johanul@chalmers.se}

\begin{document}

\begin{abstract}
We construct a positivity-preserving Lie--Trotter splitting scheme with finite difference discretization in space
for approximating the solutions to a class of nonlinear stochastic heat equations with multiplicative space-time white noise.
We prove that this explicit numerical scheme converges in the mean-square sense, with rate $1/4$ in time and rate $1/2$ in space,
under appropriate CFL conditions. Numerical experiments illustrate the superiority of
the proposed numerical scheme compared with standard numerical methods which do not preserve positivity.
\end{abstract}

\maketitle
{\small\noindent
{\bf AMS Classification.} 60H35. 60M15. 65J08.

\bigskip\noindent{\bf Keywords.} Stochastic partial differential equations.
Stochastic heat equation. Splitting scheme. Positivity-preserving scheme. Mean-square convergence.

\vspace{1cm}

\section{Introduction}\label{sec:intro}

Starting with the seminal work \cite{MR1341554} on an implicit scheme
for stochastic quasi-linear parabolic partial differential equations in $1995$,
the field of numerical analysis of stochastic partial differential equations (SPDEs)
has gained a huge interest during the last decades. We refer the interested readers
to \cite{MR876085,MR3236753,MR3222416,MR1500166,MR3308418} for references on the theory of SPDEs
and to \cite{MR1699161,MR1683281,MR1803132,MR1873517,MR1953619,MR2211047,MR2135265,MR2136207,MR2147242,MR2182132,MR2300300,MR2334046,MR2471778,MR2578878,MR2600932,
MR2891223,MR2916876,MR3027891,MR3011387,MR3101829,MR3047942,MR3308418,MR3154916,MR3364862,MR3369398,MR3327065,MR3649432,MR3761284,MR3852620,MR3906826,
MR4050540,MR4112639,MR4092279,MR4147698,MR4221304,MR4489191,bdg21,bnsz22}  for
references on the numerical analysis of SPDEs (with a particular focus on
works related to strong convergence for parabolic SPDEs).

In this work we propose and study a novel positivity-preserving numerical scheme for a fully discrete approximation of the following nonlinear Stochastic Heat Equation (SHE)
with multiplicative space-time white noise
\begin{equation}\label{eq:spde-intro}
\left\lbrace
\begin{aligned}
&\partial_t u(t,x) = \partial_{xx}^2 u(t,x) + g(u(t,x))\dot{W}(t,x),\\
&u(t,0)=u(t,1)=0,\\
&u(0,x) = u_{0}(x),
\end{aligned}
\right.
\end{equation}
for $(t,x) \in [0,T] \times [0,1]$ and where $u_{0} \geq 0$ is continuous, $g \colon \mathbb{R} \to \mathbb{R}$ is globally Lipschitz continuous, of class $\mathcal{C}^1$ and satisfies $g(0)=0$, and $\dot{W}$ is a space-time white noise, see Section~\ref{sec:setting} for precise definitions and assumptions.
Taking $g(x)=x$ in equation~\eqref{eq:spde-intro} results in the celebrated parabolic Anderson model, see for instance \cite{MR1185878}.
This equation is used to model (particle) branching processes, hydrodynamics with random forcing, and serves as a model for turbulent diffusions.

The positivity-preserving property of the exact solutions to the SPDE~\eqref{eq:spde-intro} is the subject of extensive research:
two of the first results in this direction can be found in \cite{MR1149348,MR1271224}, where this property is proven to be true for noise of the form $u^{\gamma} \dot{W}$ (where $1 \leq \gamma<3/2$)
and for a nonlinearity that is of at most linear growth. The case of a Lipschitz nonlinearity $g$ is studied in, for example, \cite{MR1500166,Ryzhik,MR3262487}.
For the sake of completeness, we mention the paper \cite{MR1443138} on positivity of SHE with random initial conditions, the paper \cite{MR1644618}
on problems with spatially homogeneous Wiener process, the paper \cite{MR3606745} on the stochastic
fractional heat equation, the paper \cite{MR3916940} on problems in $\mathbb{R}^n$, as well as the paper
\cite{MR3069916} on systems of SHEs with a spatially correlated noise. Note that these references are considering the space domain to be $\mathbb{R}$ or $\mathbb{R}^n$.
To the best of our current knowledge, there are no corresponding results for the case of compact domains with homogeneous Dirichlet boundary conditions.

While standard time integrators for SPDEs, such as the Euler--Maruyama scheme \cite{MR1803132},
the semi-implicit Euler--Maruyama scheme \cite{MR1699161}, and the stochastic exponential Euler integrator \cite{MR3047942}
do converge when applied to the problem~\eqref{eq:spde-intro}, they do not preserve the positivity property of the exact solution.
Note that the semi-implicit Euler scheme and the exponential Euler integrator preserve positivity in the
deterministic case ($g \equiv 0$ in equation~\eqref{eq:spde-intro}).

In this work, we employ a splitting strategy for the time integration of the SPDE~\eqref{eq:spde-intro}.
This results in an efficient and positivity-preserving explicit time integrator.
In essence, a splitting integrator decomposes the vector field of the
original evolution equation in several parts, such that the arising subsystems are exactly integrated (or easily). Splitting schemes have been extensively studied and successfully applied to deterministic differential equations, see for instance \cite{MR2840298,MR3642447,MR2009376} and references therein. Splitting schemes are also very popular for an efficient time discretization of stochastic (partial) differential equations.
We refer the reader to the following non-exhaustive list of articles:  \cite{m06,MR2646103,MR3021492,MR3119724,MR3617573,MR3607207,MR3736651,MR3912762,MR3839068,MR4019051,MR4132896,MR4263224,MR4278943,bc20,MR4400428,bcg22}.

The preservation of positivity by numerical methods have been investigated in several references in both the deterministic and stochastic settings.
Without being exhaustive, we mention the following articles on positivity-preserving schemes for stochastic differential equations:
\cite{MR3060559,MR2806537,MR3331648,MR4242953,Abiko,MR4335122,MR4475995,MR4489741}.
Finally, let us mention the recent reference \cite{MR4449543} on a positivity-preserving numerical scheme
for the linear stochastic heat equation with finite dimensional noise. We are not aware of works on
the numerical analysis of positivity-preserving schemes for SPDEs driven by space-time white noise.

The fully-discrete Lie--Trotter splitting scheme, see equation~\eqref{eq:scheme}, considered in this article
combines a finite difference approximation in space and the explicit recursion
$$
u_{m+1}^{\LT}=\exp\left(\tau N^2D^N\right)\hat{u}_{m+1}^{\LT},
$$
where for $n=1,\ldots,N-1$ one has
$$
\hat{u}_{m+1,n}^{\LT}=\exp\left(\sqrt{N}f(u_{m,n}^{\LT})\Delta_{m,n}W-\frac{Nf(u_{m,n}^{\LT})^2\tau}{2}\right)u_{m,n}^{\LT},
$$
where $\tau=T/M>0$ denotes the time step size, $h=1/N$ is the mesh size, $\Delta_{m,n}W$ denote space-time Wiener increments,
$N^2D^N$ the $(N-1)\times(N-1)$ matrix of the discrete Laplace operator, and $g(v)=vf(v)$.
Observe that the diffusion part of~\eqref{eq:spde-intro} is solved exactly, while
the noise part is solved exactly in the case of the parabolic Anderson model
(where one has $g(v)=v$ and $f(v)=1$ and thus the subsystem is a geometric Brownian motion).
This shares similarity with the works \cite{MR3941886,MR4356832} on stochastic differential equations. For a general mapping $g$, we
freeze the factor $f$ at the previous time point and obtain a geometric Brownian motion in the spirit
of the exponential scheme proposed in \cite{MR4177372} for finite dimensional problems.

The main results of the paper are the following:
\begin{itemize}
\item We obtain a fully discrete explicit approximation of the stochastic heat equation~\eqref{eq:spde-intro} that is positivity-preserving, see Proposition~\ref{propo:positivity-scheme}.
\item We show bounds for the second moment of the numerical approximation under a CFL condition $\tau/h={\rm O}(1)$ in Proposition~\ref{propo:moment}.
\item We prove strong convergence, with rate $1/4$, for the temporal discretization under a CFL condition $\tau/h^2={\rm O}(1)$,
see Theorem~\ref{theo:main}. The strong convergence of the fully discrete scheme is provided in Corollary~\ref{cor:errorfull}.
\end{itemize}
We leave the study of weak convergence of the proposed scheme to possible future works.
On top of that, we show positivity of the exact solution to the SPDE~\eqref{eq:spde-intro} on compact domains.
This follows naturally from the numerical analysis of the proposed approximation, see Proposition~\ref{propo:positivity-exact}.
Let us mention that the CFL conditions above are not due to the discretization of the Laplace operator, since the linear part is solved exactly.
They are due to the discretization of the contribution of the space-time white noise in the temporal evolution. Numerical experiments confirm that the CFL condition is necessary when studying the mean-square convergence of the proposed scheme.

This paper is organized as follows. Section~\ref{sec:setting} presents the setting, assumptions, and useful results on the considered SHE.
We also recall results on the finite difference discretization from \cite{MR1644183}.
Section~\ref{sec:scheme} contains the definition of the proposed Lie--Trotter splitting as well as the main results of the paper.
We postpone their proofs to Section~\ref{sec:proofs}. We dedicate Section~\ref{sec:num} to numerical experiments illustrating our qualitative and quantitative results on
the proposed splitting scheme. The last section~\ref{sec:syst} briefly presents an extension to systems of nonlinear stochastic heat equations.
Appendix~\ref{sec:app} contains a proof of an auxiliary inequality used in the proofs of the main results.

\section{Setting}\label{sec:setting}
This section provides the necessary setting for
the description of the considered class of nonlinear stochastic heat equations as well as of its solution.
We recall the notion of a mild solution and a standard well-posedness result for completeness. In addition, we recall
the spatial discretization by finite difference from \cite{MR1644183}.

For any real-valued continuous function $v\colon[0,1]\to\R$, let $\|v\|_\infty=\underset{x\in[0,1]}\max~|v(x)|$.

Let $(\Omega,\mathcal{F},\mathbb{P})$ be a probability space, equipped with a filtration $\bigl(\mathcal{F}_t\bigr)_{t\ge 0}$ which satisfies the usual conditions. The expectation operator is denoted by $\mathbb{E}[\cdot]$. In the sequel, $C$ denotes a generic constant that may vary from line to line. We sometimes use subscripts on $C$ to indicate dependence on parameters.

\subsection{Description of the SPDE}\label{sec:spde}

Let us first introduce the main assumptions needed for the numerical analysis of the proposed time integrator for the stochastic heat equation.
\begin{assumption}\label{ass:u0}
The initial value $u_0\colon[0,1]\to \R$ is a function of class $\mathcal{C}^3$, and satisfies the conditions $u_0(0)=u_0(1)=0$.
\end{assumption}
Note that the regularity assumption on the initial value above is for ease of presentation. For weaker conditions, see~\cite{MR1644183} or~\cite{MR4050540}.

When discussing positivity-preserving properties, a further condition is needed.
\begin{assumption}\label{ass:pos}
The initial value $u_0\colon[0,1]\to \R$ satisfies $u_0(x)\ge 0$ for all $x\in[0,1]$.
\end{assumption}

For the nonlinearity in the considered SPDE, we make use of the following.
\begin{assumption}\label{ass:g}
The mapping $g\colon\R\to\R$ is of class $\mathcal{C}^1$, is globally Lipschitz continuous, and satisfies $g(0)=0$.
\end{assumption}
We denote by $\Lg$ the Lipschitz constant of $g$:
\[
\Lg=\underset{v_1,v_2\in\R, v_2\neq v_1}\sup~\frac{|g(v_2)-g(v_1)|}{|v_2-v_1|}.
\]
The moment bounds and the error estimates presented below depend on the value of the Lipschitz constant $\Lg$.
This is not indicated in order to simplify the notation.

We then introduce the auxiliary mapping $f\colon\R\to\R$ defined for all $v\in\R\setminus \{ 0 \}$ by
\begin{equation}\label{nbrforCEB}
f(v)=\frac{g(v)}{v}=\int_0^1 g'(rv)\,\text dr
\end{equation}
and by $f(0)=g'(0)$. Since $g'$ is continuous by Assumption~\ref{ass:g}, the mapping $f$ is continuous and bounded, and one has the upper bound $\underset{v\in\R}\sup~|f(v)|\le \Lg$.

For a fixed time horizon $T>0$, let $W=\left\{W(t,x)\,\colon\, t\in[0,T], x\in[0,1]\right\}$ be an $\mathcal{F}_t$-adapted Brownian sheet.
We recall that a Brownian sheet is a Gaussian random field with mean zero and covariance $\E[W(t,x)W(s,y)]=(t\wedge s)(x\wedge y)$ for all $s,t\in[0,T]$ and
$x,y\in[0,1]$, see for instance \cite{MR876085}. We consider the stochastic heat equation in the It\^o sense
\begin{equation}\label{eq:spde}
\left\lbrace
\begin{aligned}
&\text d u(t,x) = \partial_{xx}^2 u(t,x)\,\text dt + g(u(t,x))\,\text dW(t,x),\\
&u(t,0)=u(t,1)=0,\\
&u(0,x) = u_{0}(x)
\end{aligned}
\right.
\end{equation}
for $t\in[0,T]$ and $x\in[0,1]$, where $u_{0}$ and $g$ satisfy Assumption~\ref{ass:u0} and Assumption~\ref{ass:g}, respectively.

In order to define a mild solution of the stochastic heat equation~\eqref{eq:spde}, we introduce the heat kernel
\begin{equation*}
G(t,x,y) = \sum_{j=1}^{\infty} e^{- j^{2} \pi^{2} t} \sin(j \pi x) \sin(j \pi y),
\end{equation*}
for $t\ge 0,x,y\in[0,1]$, which is the fundamental solution of the (deterministic) heat equation with homogeneous Dirichlet boundary conditions:
\begin{equation*}
\left\lbrace
\begin{aligned}
&\text d v(t,x) = \partial_{xx}^2 v(t,x)\,\text dt,\\
&v(t,0)=v(t,1)=0,\\
&v(0,x) = \delta (x),
\end{aligned}
\right.
\end{equation*}
where the initial value is the Dirac delta function.

A mild solution to the SPDE~\eqref{eq:spde} is a random field $\bigl(u(t,x)\bigr)_{t \in [0,T],x \in [0,1]}$
satisfying the following integral equation almost surely: for all $t\in[0,T]$ and $x\in[0,1]$, one has
\begin{equation}\label{eq:spde-mild}
u(t,x)=\int_0^1 G(t,x,y)u_0(y)\,\text dy+\int_0^t \int_0^1 G(t-s,x,y)g(u(s,y))\,\text dW(s,y).
\end{equation}
The stochastic integral in~\eqref{eq:spde-mild} is understood in the It\^o--Walsh sense, see for instance \cite{MR1500166,MR3222416,MR876085}.

We collect some properties of the mild solution $u(t,x)$ to the stochastic heat equation~\eqref{eq:spde} in the following statement, see for instance~\cite[Proposition~3.7]{MR1644183}.
\begin{proposition}
Consider the stochastic heat equation~\eqref{eq:spde} under Assumptions~\ref{ass:u0}~and~\ref{ass:g}.
There exists a unique mild solution $\left(u(t,x)\right)_{t\in[0,T],x\in[0,1]}$  to the SPDE~\eqref{eq:spde}.
In addition, for all $T \in (0,\infty)$, there exists $C_T\in(0,\infty)$ such that
\[
\underset{t\in[0,T]}\sup~\underset{x\in[0,1]}\sup~\E[|u(t,x)|^2]\le C_T(1+\|u_0\|_\infty^2).
\]
Finally, the solution satisfies the following mean-square regularity property: for all $T\in(0,\infty)$,
there exists $C_{T}\in(0,\infty)$ such that for all $x_1,x_2\in[0,1]$ and all $t_1,t_2\in[0,T]$ one has
\begin{equation}\label{eq:regulexact}
\left(\E[|u(t_2,x_2)-u(t_1,x_1)|^2]\right)^{\frac12}\le C_{T}\bigl(|t_2-t_1|^{\frac14}+|x_2-x_1|^{\frac12}\bigr).
\end{equation}
\end{proposition}

In this article, our objective is to propose and analyze consistent numerical schemes which preserve the following property of the exact solution:
if the initial value $u_0$ is nonnegative, then the exact solution to the stochastic heat equation, $u(t,\cdot)$, remains nonnegative for all $t>0$.
\begin{proposition}\label{propo:positivity-exact}
Consider the stochastic heat equation~\eqref{eq:spde} together with Assumptions~\ref{ass:u0}, \ref{ass:pos} and~\ref{ass:g}.
Then, for all $t\in(0,\infty)$ and all $x\in[0,1]$, almost surely, one has
\[
u(t,x)\ge 0.
\]
\end{proposition}
The proof of Proposition~\ref{propo:positivity-exact} above is postponed to Section~\ref{sec:proof-positivity-exact}.
It is a consequence of the analysis of the fully-discrete numerical scheme and combines two arguments:
on the one hand, the numerical scheme satisfies a variant of Proposition~\ref{propo:positivity-exact}, see Proposition~\ref{propo:positivity-scheme} below,
on the other hand, Theorem~\ref{theo:main} gives a strong convergence result of the numerical approximation.
Note that similar results are known when considering the stochastic heat equation on the real line, see for instance the works
\cite{MR1149348,MR1271224} and the lecture notes \cite{Ryzhik}. We are not aware of positivity-preserving results for SPDEs on bounded domains.

\subsection{Spatial discretization}\label{sec:spatial-discretization}
Let us recall the spatial discretization based on a finite difference approximation on a uniform grid from \cite{MR1644183}. For any integer $N\in\N$, let $h=1/N$ be the space mesh size,
and let $x_n=nh$ for $0\le n\le N$ be the grid points. Let $\kappa^N\colon[0,1]\to\{x_0,\ldots,x_N\}$,
be the mapping defined by $\kappa^{N}(x)=x_n$ for $x\in[x_n,x_{n+1})$ if $n\in\{0,\ldots,N-1\}$, and $\kappa^N(1)=\kappa^{N}(x_N)=x_N=1$.

Throughout this article, we use the convention that for any vector $v=\bigl(v_n\bigr)_{1\le n\le N-1}\in\R^{N-1}$, we append discrete homogeneous Dirichlet boundary conditions $v_{0}=0$ and $v_{N}=0$ when needed.

We discretize the initial value $u_0$ of the stochastic heat equation~\eqref{eq:spde} by $u_{0,n}^N=u_n^N(0)=u(0,x_n)$ for $0\le n\le N$. Note that discrete homogeneous Dirichlet boundary conditions $u_{0,0}^N=u_{0,N}^N=0$ are satisfied owing to Assumption~\ref{ass:u0}. Let us then define a piecewise linear extension $u^N(0,\cdot)\colon[0,1]\to\R$ satisfying $u^N(0,x_n)=u_{0,n}^{N}$ for all $n=0,\ldots,N$, meaning that for $x \in (0,1)$ one has
\[
u^N(0,x)=N\bigl(\kappa^N(x)+h-x\bigr)u(0,\kappa^N(x))+N\bigl(x-\kappa^N(x)\bigr)u(0,\kappa^N(x)+h).
\]
Let $D^N=\bigl(D_{ij}^N\bigr)_{1\le i,j\le N-1}$ denote the matrix coming from a standard finite difference discretization of the Laplace operator
at the grid points $x_n$ with homogeneous Dirichlet boundary conditions. The matrix $D^N$ is thus given by
\[
D^N=
\begin{pmatrix}
-2 & 1 & 0 & \ldots & 0 & 0 & 0\\
1 & -2 & 1 & \ddots & 0 & 0 & 0\\
0 & 1 & -2 & \ddots & 0 & 0 & 0\\
\vdots & \ddots & \ddots & \ddots & \ddots & \ddots & \vdots\\
0 & 0 & 0 & \ddots & -2 & 1 & 0\\
0 & 0 & 0 & \ddots & 1 & -2 & 1\\
0 & 0 & 0 & \ldots & 0 & 1 & -2
\end{pmatrix}
.
\]
We then introduce the discrete heat kernel $G^N(t)=\bigl(G_{ij}^{N}(t)\bigr)_{1\le i,j\le N-1}=e^{tN^2D^N}$, for $t\geq 0$.
By convention, set $G_{00}^N(t)=G_{NN}^N(t)=1$, $G_{0N}^N(t)=G_{N0}^N(t)=0$ and $G_{0j}^N(t)=G_{Nj}^N(t)=0$ for all $j\in\{1,\ldots,N-1\}$, in order to satisfy homogeneous discrete Dirichlet boundary conditions. Finally, we extend the definition of $G^N(t)=\bigl(G_{ij}^{N}(t)\bigr)_{1\le i,j\le N-1}$ to $\bigl(G^N(t,x,y)\bigr)_{t\ge 0,x,y\in[0,1]}$ by asking
that $G^N(t,x_i,y_j)=NG_{ij}^N(t)$ for $0\le i,j\le N$ and for $t\ge 0$ and $y\in[0,1]$
\[
G^N(t,x,y)=N\bigl(\kappa^N(x)+h-x\bigr)G^{N}(t,\kappa^N(x),\kappa^N(y))+N\bigl(x-\kappa^N(x)\bigr)G^N(t,\kappa^N(x)+h,\kappa^N(y))
\]
for $x\in(0,1)$ and $G^N(t,0,y)=G^N(t,1,y)=0$. As a result, the mapping $(x,y)\mapsto G^N(t,x,y)$ is piecewise linear in $x$ and piecewise constant in $y$ at all times $t$.

It is worth recalling the following well-known property of the discrete heat kernel: one has $G_{ij}^{N}(t)\ge 0$ for all $i,j\in\{1,\ldots,N-1\}$ and $t\ge 0$. As a consequence, one has $G^N(t,x,y)\ge 0$ for $x,y\in[0,1]$ and $t\ge 0$. This property follows from the fact that $t N^{2} D^{N}$ is a Metzler matrix,
see for instance~\cite{FarinaLorenzo2000PLST} for a definition, and the exponential of a Metzler matrix has only non-negative elements.

\bigskip

We are now in position to define the spatial discretization $u^{N}$, for $N\in\N$, by the following integral equality
\begin{equation}\label{eq:uN}
u^N(t,x)=\int_0^1 G^N(t,x,y)u^N(0,\kappa^N(y))\,\text dy+\int_0^t\int_0^1 G^N(t-s,x,y)g(u^N(s,\kappa^N(y)))\,\text dW(s,y)
\end{equation}
for $t\ge 0$ and $x\in[0,1]$.
Note that the mapping $x\in[0,1]\mapsto u^N(t,x)$ is linear on $[x_n,x_{n+1}]$, for every $n\in\{0,\ldots,N-1\}$, for every $t\ge 0$. In addition, one has $u^N(t,0)=u^N(t,1)=0$ for every $t\ge 0$. For a practical implementation of the scheme, it is sufficient to compute $u_{n}^N(t)=u^N(t,x_n)$ for all $1\le n\le N-1$. This is performed as follows: for all $t\ge 0$ and $1\le n\le N-1$ one has
\[
u_n^N(t)=\sum_{j=1}^{N-1}G_{nj}^N(t)u_{0,j}^N+\sqrt{N}\sum_{j=1}^{N-1}\int_0^t G_{nj}^N(t-s)g(u_j^N(s))\,\text dW_j^N(s),
\]
where
\[
W_n^N(t)=\sqrt{N}\bigl(W(t,x_{n+1})-W(t,x_n)\bigr).
\]
By definition of a Wiener sheet, observe that the processes $\bigl(W_1^N(t)\bigr)_{t\ge 0},\ldots,\bigl(W_{N-1}^N(t)\bigr)_{t\ge 0}$ are independent standard real-valued Wiener processes, for any $N\in\N$.

Introduce the $\R^{N-1}$-valued process $u^N$ defined by $u^N(t)=\bigl(u_n^N(t)\bigr)_{1\le n\le N-1}$ for all $t\ge 0$. This process is solution of the following stochastic differential equation
\begin{equation}\label{eq:spatialscheme}
\text du^N(t)=N^2D^Nu^N(t)\,\text dt+\sqrt{N}g(u^N(t))\,\text dW^N(t)
\end{equation}
with initial value $u^N(0)=\bigl(u_0^N\bigr)_{1\le n\le N-1}$, where the notation
$\bigl(g(u^N(t))\,\text dW^N(t)\bigr)_n=g(u_n^N(t))\,\text dW_n^N(t)$ is used.

Let us recall the following convergence result for the spatial discretization, see \cite[Theorem~3.1]{MR1644183}.
\begin{proposition}\label{propo:error-spatial}
Consider the stochastic heat equation~\eqref{eq:spde} with a nonlinearity $g$ satisfying Assumption~\ref{ass:g}.
Denote by $\left(u(t,x)\right)_{t\in[0,T],x\in[0,1]}$ its exact solution and by $\left(u^N(t,x)\right)_{t\in[0,T],x\in[0,1]}$ the numerical approximation by finite differences with mesh size $h=1/N$.
For all $T\in(0,\infty)$ and any initial value $u_0$ satisfying Assumption~\ref{ass:u0}, there exists $C_{T}(u_0)\in(0,\infty)$ such that for all $h=1/N$ with $N\in\N$ one has
\begin{equation}\label{eq:errorspatial}
\underset{t\in[0,T]}\sup~\underset{x\in[0,1]}\sup~\bigl(\E[|u^N(t,x)-u(t,x)|^2]\bigr)^{\frac12}\le C_{T}(u_0)h^{\frac12}.
\end{equation}
\end{proposition}

In the error analysis below, the following auxiliary result from \cite{MR4050540} (see Proposition~$2.4$) on the temporal regularity of $u^N$ is used:
there exists $C_{T}(u_0)\in(0,\infty)$ such that for all $t,s\in[0,T]$, one has
\begin{equation}\label{eq:regul-uN}
\underset{N\in\N}\sup~\underset{x\in[0,1]}\sup~\E[|u^N(t,x)-u^N(s,x)|^2]\le C_{T}(u_0)|t-s|^{\frac12}.
\end{equation}

\section{The positivity-preserving splitting scheme}\label{sec:scheme}
In the core part of this paper, we present and study the strong convergence of
an efficient and positivity-preserving time integrator for the stochastic heat equation~\eqref{eq:spde}.

Let $T\in(0,\infty)$ and divide the interval $[0,T]$ into $M\in\N$ subintervals $[t_m,t_{m+1}]$ of length $\tau=T/M$, where $t_m=m\tau$ for $m\in\{0,\ldots,M\}$.
Introduce the mapping $\ell^M\colon[0,T]\to\{t_0,\ldots,t_M\}$, defined by $\ell^{M}(t)=t_m$ for all $t\in[t_m,t_{m+1})$, if $m\in\{0,\ldots,M-1\}$, and $\ell^{M}(T)=\ell^M(t_M)=t_M=T$.

We propose a fully-discrete explicit scheme based on a Lie--Trotter splitting strategy
producing approximations $u_m^{\LT}=\bigl(u_{m,n}^{\LT}\bigr)_{1\le n\le N-1}$ of the finite difference approximation
$u^N(t_m)=\bigl(u_n^{N}(t_m)\bigr)_{1\le n\le N-1}$ at the grid times $t_m$, $m=0,\ldots,M$. We set the initial value to be $u_{0,n}^{\LT}=u_n^N(0)=u_{0,n}^N$ for all $1\le n\le N-1$. As above, one has $u_{m,0}^{\LT}=0$ and $u_{m,N}^{\LT}=0$ for all $m\in\{0,\ldots,M\}$. In this way, homogeneous Dirichlet boundary conditions are satisfied by the numerical scheme at all times.

We explain the construction of the scheme in Section~\ref{sec:schemeconstruction}. We then describe the main results of this article: the positivity-preserving property of the splitting scheme (Proposition~\ref{propo:positivity-scheme}) and the mean-square convergence in time with order $1/4$ (Theorem~\ref{theo:main} and Corollary~\ref{cor:errorfull}).

\subsection{Description of the time integrator}\label{sec:schemeconstruction}

Let us describe how the splitting scheme is constructed.  Given the numerical solution $u_{m}^{\LT}=\left(u_{m,n}^{\LT}\right)_{1\le n\le N-1}$ at grid time $t_m=m\tau$ for $0\le m\le M-1$,
the solution $u_{m+1}^{\LT}$ at the next grid time $t_{m+1}=t_m+\tau$ is constructed by successively solving two subsystems in $\R^{N-1}$:
\begin{itemize}
\item first, the linear It\^o SDE system
\begin{equation}\label{eq:scheme-subsystem1}
\text dv_{m,n}^{M,N,1}(t)=\sqrt{N}v_{m,n}^{M,N,1}(t)f(u_{m,n}^{\LT})\,\text dW_n^N(t),
\end{equation}
for $n\in\{1,\ldots,N-1\}$ and $t\in[t_m,t_{m+1}]$, with initial value $v_{m,n}^{M,N,1}(t_m)=u_{m,n}^{\LT}$, where we recall
(see equation~\eqref{nbrforCEB} in Section~\ref{sec:setting}) that the auxiliary function $f$ is such that $g(v)=vf(v)$ for all $v\in\R$;
\item second, the linear ODE system
\begin{equation}\label{eq:scheme-subsystem2}
\text dv_{m}^{M,N,2}(t)=N^2D^Nv_{m}^{M,N,2}(t)\,\text dt,
\end{equation}
for $t\in[t_m,t_{m+1}]$, with initial value $v_{m,n}^{M,N,2}(t_m)=v_{m,n}^{M,N,1}(t_{m+1})$.
\end{itemize}
Observe that the solutions of the two subsystems above are known: the solution of the SDE~\eqref{eq:scheme-subsystem1} is given by
\begin{equation}\label{eq:scheme-subsystem1-solution}
v_{m,n}^{M,N,1}(t)=\exp\left(\sqrt{N}f(u_{m,n}^{\LT})\bigl(W_n^N(t)-W_n^N(t_m)\bigr)-\frac{Nf(u_{m,n}^{\LT})^2 (t-t_m)}{2}\right)u_{m,n}^{\LT},
\end{equation}
for all $t\in[t_m,t_{m+1}]$, and the solution of the ODE~\eqref{eq:scheme-subsystem2} is given by
\begin{equation}\label{eq:scheme-subsystem2-solution}
v_{m}^{M,N,2}(t)=e^{(t-t_m)N^2D^N}v_{m}^{M,N,1}(t_{m+1}),
\end{equation}
for all $t\in[t_m,t_{m+1}]$.

Gathering the expressions above gives the following expression for the proposed Lie--Trotter splitting scheme
\begin{equation}\label{eq:scheme}
u_{m+1}^{\LT}=e^{\tau N^2D^N}\left(\exp\Bigl(\sqrt{N}f(u_{m,n}^{\LT})\Delta_{m,n}W-\frac{Nf(u_{m,n}^{\LT})^2\tau}{2}\Bigr)u_{m,n}^{\LT}\right)_{1\le n\le N-1},
\end{equation}
where $\Delta_{m,n}W=W_n^{N}(t_{m+1})-W_n^{N}(t_m)$. Observe that the random variables
$\bigl(\Delta W_{m,n}\bigr)_{0\le m\le M-1,1\le n\le N-1}$ are independent standard real-valued Gaussian random variables.

The splitting scheme formula~\eqref{eq:scheme} can also be written as
\begin{equation*}
u_{m+1,n}^{\LT}=\sum_{k=1}^{N-1}G_{nk}^{N}(\tau)\exp\Bigl(\sqrt{N}f(u_{m,k}^{\LT})\Delta_{m,n}W-\frac{N f(u_{m,k}^{\LT})^2\tau}{2}\Bigr)u_{m,k}^{\LT},
\end{equation*}
for all $m\in\{0,\ldots,M-1\}$ and $n\in\{1,\ldots,N-1\}$.

One of the key properties of the proposed splitting scheme is the following: if the initial value $u_0^{\LT}=\bigl(u_{0,n}^{\LT}\bigr)_{1\le n\le N-1}$ only has nonnegative elements, then for all $m\in\{1,\ldots,M\}$ the numerical solution $u_m^{\LT}=\bigl(u_{m,n}^{\LT}\bigr)_{1\le n\le N-1}$ at time $t_m=m\tau$ also only has nonnegative elements almost surely. In other words, the proposed scheme is positivity-preserving. This is stated in the next proposition.

\begin{proposition}\label{propo:positivity-scheme}
Let $M\in\N$ and $N\in\N$ be arbitrary integers and let $T\in(0,\infty)$.
Let Assumptions~\ref{ass:u0},~\ref{ass:pos} and~\ref{ass:g} be satisfied.
Let the sequence $u_{0}^{\LT},\ldots,u_{M}^{\LT}$ be given by the splitting scheme~\eqref{eq:scheme}, with $h=1/N$ and $\tau=T/M$, with initial value $u_{0,n}^{\LT}=u_0(x_n)\ge 0$ for all $n\in\{1,\ldots,N\}$. Then, almost surely, one has
\[
u_{m,n}^{\LT}\ge 0,
\]
for all $m\in\{1,\ldots,M\}$ and $n\in\{1,\ldots,N-1\}$.
\end{proposition}

\begin{proof}[Proof of Proposition~\ref{propo:positivity-scheme}]
The proof proceeds by recursion on the time index $m$.
\begin{itemize}
\item Note that $u_{0,n}^{\LT}=u(0,x_n)\ge 0$ for all $n\in\{0,\ldots,N\}$.
\item Assume that the property $u_{m,n}^{\LT}\ge 0$, for all $n\in\{1,\ldots,N-1\}$, holds at time $t_{m} = m \tau$. We prove that under this assumption, it also holds at time $t_{m+1}=(m+1) \tau$.

The argument is straightforward: the solutions of the subsystems~\eqref{eq:scheme-subsystem1} and~\eqref{eq:scheme-subsystem2} are nonnegative at all times when they have nonnegative initial values. More precisely, first one has
\[
v_{m,n}^{M,N,1}(t_{m+1})=e^{\sqrt{N}f(u_{m,n}^{\LT})\Delta_{m,n}W-N\frac{f(u_{m,n}^{\LT})^2 \tau}{2}}u_{m,n}^{\LT}\ge 0
\]
for all $n\in\{1,\ldots,N-1\}$. Second, using the inequality $G_{nk}^N(\tau)\ge 0$ (see Section~\ref{sec:spatial-discretization}), one has
\[
u_{m+1,n}^{\LT}=v_{m,n}^{M,N,2}(t_{m+1})=\sum_{k=1}^{N-1}G_{nk}^{N}(\tau)v_{m,k}^{M,N,1}(t_{m+1})\ge 0.
\]
Thus the positivity property of the numerical solution holds at time $t_{m+1}=(m+1) \tau$.
\end{itemize}
As a consequence, the property $u_{m,n}^{\LT}\ge 0$, for all $n\in\{1,\ldots,N-1\}$, holds for any $m\in\{0,\ldots,M\}$. The proof of Proposition~\ref{propo:positivity-scheme} is completed.
\end{proof}

\subsection{Convergence results}\label{sec:results}

Let us now prove that the proposed numerical scheme provides accurate approximation of the exact solution.
In this article, we show mean-square error estimates and give orders of convergence with respect to $\tau=T/M$ and $h=1/N$.

We impose a CFL stability condition in the sequel to ensure stability and convergence of the Lie--Trotter splitting scheme~\eqref{eq:scheme}
when applied to the stochastic heat equation~\eqref{eq:spde}; more precisely, we introduce conditions of the type $\tau \le \gamma h$ or $\tau\le \gamma h^2$ in the statements below, for some (nonrandom) arbitrary parameter $\gamma\in(0,\infty)$. The conditions on $\tau$ and $h$ above are equivalent to the conditions $\gamma M\ge TN$ and $\gamma M\ge TN^2$ on $M$ and $N$  respectively.

Owing to Proposition~\ref{propo:error-spatial}, it is sufficient to focus on the error $u_{m,n}^{\LT}-u_n^{N}(t_m)$ to obtain estimates for the total error $u_{m,n}^{\LT}-u(t_m,x_n)$. Proposition~\ref{propo:moment} shows moment bounds of the numerical solution and
is used to prove our main result in Theorem~\ref{theo:main}. As a corollary we obtain convergence of $u^{\LT}_{m,n}$ to the exact solution $u(t_{m},x_{n})$
at the grid points using results from \cite{MR1644183}.

Note that Assumption~\ref{ass:pos} on the positivity of the initial value is not needed in the statements on the moment bounds and on the convergence of the scheme below.

\begin{proposition}\label{propo:moment}
Assume that Assumptions~\ref{ass:u0} and~\ref{ass:g} are satisfied. Let the sequence $u_{0}^{\LT},\ldots,u_{M}^{\LT}$ be given by the Lie--Trotter splitting scheme~\eqref{eq:scheme}.

For all $\gamma\in(0,\infty)$ and all $T\in(0,\infty)$, there exists $C_{\gamma,T}\in(0,\infty)$ such that for all $\tau=T/M$ and $h=1/N$ satisfying the condition $\tau \le \gamma h$, one has
\begin{equation}\label{eq:momentscheme}
\underset{0\le m\le M}\sup~\underset{1\le n\le N-1}\sup~\E\left[|u_{m,n}^{\LT}|^2\right]\le C_{\gamma,T}\bigl(1+\|u_0\|_\infty^2\bigr).
\end{equation}
\end{proposition}
The proof of this proposition also provides moment bounds for a space-time continuous version $u^{\LT}(t,x)$,
defined by equation~\eqref{eq:schemeaux3}, of the Lie--Trotter splitting scheme~\eqref{eq:scheme}.

We are now in position to state the main convergence result of this article. For ease of presentation, we only consider errors at space-time grid points.

\begin{theorem}\label{theo:main}
Assume that Assumptions~\ref{ass:u0} and~\ref{ass:g} are satisfied. Let the sequence $u_{0}^{\LT},\ldots,u_{M}^{\LT}$ be given by the Lie--Trotter scheme~\eqref{eq:scheme}, and let $\bigl(u^N(t)\bigr)_{t\ge 0,0\le n\le N}$ be given by the spatial semi-discretization scheme~\eqref{eq:spatialscheme}.

For all $\gamma\in(0,\infty)$ and $T\in(0,\infty)$, there exists $C_{\gamma,T}(u_0)\in(0,\infty)$ such that for all $\tau=T/M$ and $h=1/N$ satisfying the condition $\tau\le \gamma h$, one has
\begin{equation}\label{eq:error}
\underset{0\le m\le M}\sup~\underset{0\le n\le N}\sup~\left(\E[|u_{m,n}^{\LT}-u_n^N(t_m)|^2]\right)^{\frac12}\le C_{\gamma,T}(u_0)\left(\tau^{\frac{1}{4}}
+\left(\frac\tau h\right)^{\frac12}\right).
\end{equation}
In addition, for all $\tau=T/M$ and $h=1/N$ satisfying the condition $\tau\le \gamma h^2$, one has
\begin{equation}\label{eq:error2}
\underset{0\le m\le M}\sup~\underset{0\le n\le N}\sup~\left(\E[|u_{m,n}^{\LT}-u_n^N(t_m)|^2]\right)^{\frac12}\le C_{\gamma,T}(u_0)\tau^{\frac{1}{4}}.
\end{equation}
\end{theorem}

Proving~\eqref{eq:error2} from the error estimate~\eqref{eq:error} under the stronger condition $\tau\le \gamma h^2$ is straightforward.

Combining Theorem~\ref{theo:main} and Proposition~\ref{propo:error-spatial}, one directly obtains error estimates for the fully-discrete scheme.

\begin{corollary}\label{cor:errorfull}
Consider the setting and assumptions of Theorem~\ref{theo:main}.
For all $\gamma\in(0,\infty)$ and $T\in(0,\infty)$, there exists $C_{\gamma,T}(u_0)\in(0,\infty)$ such that for all $\tau=T/M$ and $h=1/N$ satisfying the condition $\tau\le \gamma h^2$, one has
\begin{equation}\label{eq:errorfull}
\underset{0\le m\le M}\sup~\underset{0\le n\le N}\sup~\left(\E[|u_{m,n}^{\LT}-u(t_m,x_n)|^2]\right)^{\frac12}\le C_{\gamma,T}(u_0)h^{\frac{1}{2}}.
\end{equation}
\end{corollary}

We postpone the proofs of the above results to Section~\ref{sec:proofs}.

\section{Numerical experiments}\label{sec:num}
In this section we provide numerical experiments to support and verify the above theoretical results.
Recall that $\tau=T/M>0$ is the time step size and $h = 1/N > 0$ is the space mesh size.
We compare the proposed Lie--Trotter splitting scheme~\eqref{eq:scheme}, denoted LT below,
to the following classical time integrators when applied to the spatially discretized system~\eqref{eq:spatialscheme}:
\begin{itemize}
\item the Euler--Maruyama scheme (denoted EM below), see for instance \cite{MR1803132}
$$
u^{\rm EM}_{m+1}=u^{\rm EM}_m+\tau N^2D^Nu^{\rm EM}_m+\sqrt{N}g(u^{\rm EM}_m)\Delta_mW,
$$
\item the semi-implicit Euler--Maruyama scheme (denoted SEM below), see for instance \cite{MR1699161}
$$
u^{\rm SEM}_{m+1}=u^{\rm SEM}_m+\tau N^2D^Nu^{\rm SEM}_{m+1}+\sqrt{N}g(u^{\rm SEM}_m)\Delta_mW,
$$
\item the stochastic exponential Euler integrator (denoted SEXP below), see for instance \cite{MR3047942}
$$
u^{\rm SEXP}_{m+1}=e^{\tau N^2D^N}\left(u^{\rm SEXP}_m+\sqrt{N}g(u^{\rm SEXP}_m)\Delta_mW\right).
$$
\end{itemize}

\subsection{Preservation of the positivity}

We start by illustrating the positivity-preserving property of the Lie--Trotter scheme (LT) and show the lack of positivity-preserving
behavior for the Euler--Maruyama scheme (EM), the semi-implicit Euler--Maruyama scheme (SEM), and the stochastic exponential scheme (SEXP).
To do this, we use the same noise samples for all time integrators when applied to the space-discretization of the SPDE~\eqref{eq:spde} as described in Section~\ref{sec:spatial-discretization} with the initial condition $u_{0}=\sin(\pi x)$ and final time $T=20$. We consider this problem with
the three choices of multiplicative term given by $g(v)=\lambda v$,
$g(v)= \lambda \ln \left( 1 + v \right)$, and $g(v) = \lambda \left(v + \sin(v) \right)$. The real-valued parameter $\lambda>0$ is introduced
to avoid the need to run numerical experiments with very long time horizons $T$ in order to obtain negative values for the numerical schemes SEXP and SEM.
We remark that $g(v)= \lambda \ln \left( 1 + v \right)$ is well-behaved for $v \geq 0$ but problems may occur if $v \leq -1$.
Since the proposed LT scheme is guaranteed to preserve positivity, this is not problematic. However,
this could happen for the time integrators SEXP, SEM or EM.
The numerical results are presented in Tables~\ref{table1}~and~\ref{table2}, where the notation
$k/50$ indicates that $k$ out of $50$ samples remain positive.

In Table~\ref{table1}, we let $g(v)=2.5v$ and we consider $50$ sample paths for each of the time integrators
for several choices of the discretization parameters
$\tau$ and $h$. Table~\ref{table1} confirms that the LT scheme preserves positivity. This is not the case for SEXP, SEM and EM.
We observe that fewer samples of SEXP and SEM contain negative values for small time steps $\tau$.
This is expected as each of the time integrators SEXP, SEM, and even EM, converges (for every fixed $h$) to the exact,
everywhere positive, solution of the space-discretized system of SDEs in equation~\eqref{eq:spatialscheme}.

\begin{table}[h]
\begin{center}
\begin{tabular}{|c c c c c|}
 \hline
 $(\tau,h)$ & LT & SEXP & SEM & EM\\
 \specialrule{.1em}{.05em}{.05em}
 $(10^{-3},10^{-2})$ & $50/50$ & $0/50$ & $0/50$ & $0/50$ \\
 \hline
 $(10^{-4},10^{-3})$ & $50/50$ & $50/50$ & $50/50$ & $0/50$ \\
 \hline
 $(10^{-5},10^{-3})$ & $50/50$ & $50/50$ & $50/50$ & $0/50$ \\
 \hline
 %$(10^{-6},10^{-3})$ & $50/50$  & $50/50$ & Time out ($50$h) & $0/50$ \\
 %\hline
\end{tabular}
\end{center}
\caption{Proportion of samples containing only positive values out of $50$ simulated sample paths for
the Lie--Trotter splitting scheme (LT), the stochastic exponential Euler integrator (SEXP),
the semi-implicit Euler--Maruyama scheme (SEM), and Euler--Maruyama scheme (EM) for
the diffusion coefficient $g(v)=2.5v$ and several choices of discretization parameters $\tau$ and $h$.}\label{table1}
\end{table}

In Table~\ref{table2} we instead fix the discretization parameters $\tau=10^{-5}$ and $h=10^{-3}$ and consider different types of multiplicative terms $g(v)$.
We again use $50$ samples in each of the entries of Table~\ref{table2}. From the results of Table~\ref{table2}, one can observe the poor performance of the EM scheme
in all cases. This table also illustrates the fact that increasing the size of the multiplicative term prevents SEM and SEXP to remain positive.
It should be clear that increasing the value of $\lambda$ even more, or the length of the time interval, would hinder the numerical solutions
to stay positive for all time integrators except for the proposed Lie--Trotter splitting scheme.

\begin{table}[h]
\begin{center}
\begin{tabular}{|c c c c c|}
 \hline
 $g(v)$ & LT & SEXP & SEM & EM \\
 \specialrule{.1em}{.05em}{.05em}
 $2.5 \ln(1+v) $ & $50/50$ & $50/50$ & $50/50$ & $0/50$ \\
 \hline
 $3.5 \ln(1+v) $ & $50/50$ & $50/50$ & $50/50$ & $0/50$ \\
 \hline
 $5 \ln(1+v) $ & $50/50$ & $47/50$ & $26/50$ & $0/50$ \\
 \hline
 $2.5 v $ & $50/50$ & $50/50$ & $50/50$ & $0/50$ \\
 \hline
 $3.5 v $ & $50/50$ & $50/50$ & $50/50$ & $0/50$ \\
 \hline
 $5 v $ & $50/50$ & $4/50$ & $50/50$ & $0/50$ \\
 \hline
 $2.5 \left( v + \sin(v) \right)$ & $50/50$ & $44/50$ & $50/50$ & $0/50$ \\
 \hline
 $3.5 \left( v + \sin(v) \right)$ & $50/50$ & $0/50$ & $0/50$ & $0/50$ \\
 \hline
 $5 \left( v + \sin(v) \right)$ & $50/50$ & $0/50$ & $0/50$ & $0/50$ \\
 \hline
\end{tabular}
\end{center}
\caption{Proportion of samples containing only positive values out of $50$ simulated sample paths for
the Lie--Trotter splitting scheme (LT), the stochastic exponential Euler integrator (SEXP),
the semi-implicit Euler--Maruyama scheme (SEM), and the Euler--Maruyama scheme (EM)
for several choices of diffusion terms $g(v)$. The discretization parameters are $\tau=10^{-5}$ and $h=10^{-3}$.}\label{table2}
\end{table}

\subsection{Mean-square errors}
For the next numerical experiment, we discretize the stochastic heat equation~\eqref{eq:spde} with initial value $u_0(x)=\sin(\pi x)$
by a finite-difference scheme in space with mesh size $h=2^{-8}$. The resulting system of stochastic differential equations~\eqref{eq:spatialscheme} is then
discretized by the time integrators LT, SEXP, and SEM. The classical EM scheme is not appropriate in this setting and numerical results are thus not presented.
The following choices for the function $g$ are considered:
$g(v)=v$ and $g(v)=\frac{v}{1+v^2}$ and $g(v)=\ln(1+v)$, for $v\ge0$, and $g(v)=v\exp(-v^2)$.
Figure~\ref{fig:strong1} displays, in a loglog plot, the mean-square errors
$$
\underset{0\le m\le M}\sup~\underset{0\le n\le N}\sup~\bigl(\E[|u_{m,n}^{\text{num}}-u^{\text{ref}}(t_m,x_n)|^2]\bigr)^{\frac12}
$$
measured at the space-time grid points $(t_m,x_n)$ for the time interval $[0,0.5]$. The reference solution $u^{\text{ref}}$ is computed using the LT splitting scheme
with time step size $\tau_{\text{ref}}=2^{-16}$. Here, $200$ samples have been used to approximate the expectations. We have
checked that the Monte Carlo error is negligible to observe mean-square convergence. In this figure, one can observe a rate of convergence $1/2$ instead of $1/4$ in the mean-square error estimates~\eqref{eq:error2} for the splitting scheme in Theorem~\ref{theo:main}. This is related to the mean-square error estimates~\eqref{eq:error} and the role of the CFL condition
$\tau\le \gamma h^2$ to obtain~\eqref{eq:error2}.

\begin{figure}[h]
\begin{subfigure}{.35\textwidth}
  \centering
  \includegraphics[width=\textwidth]{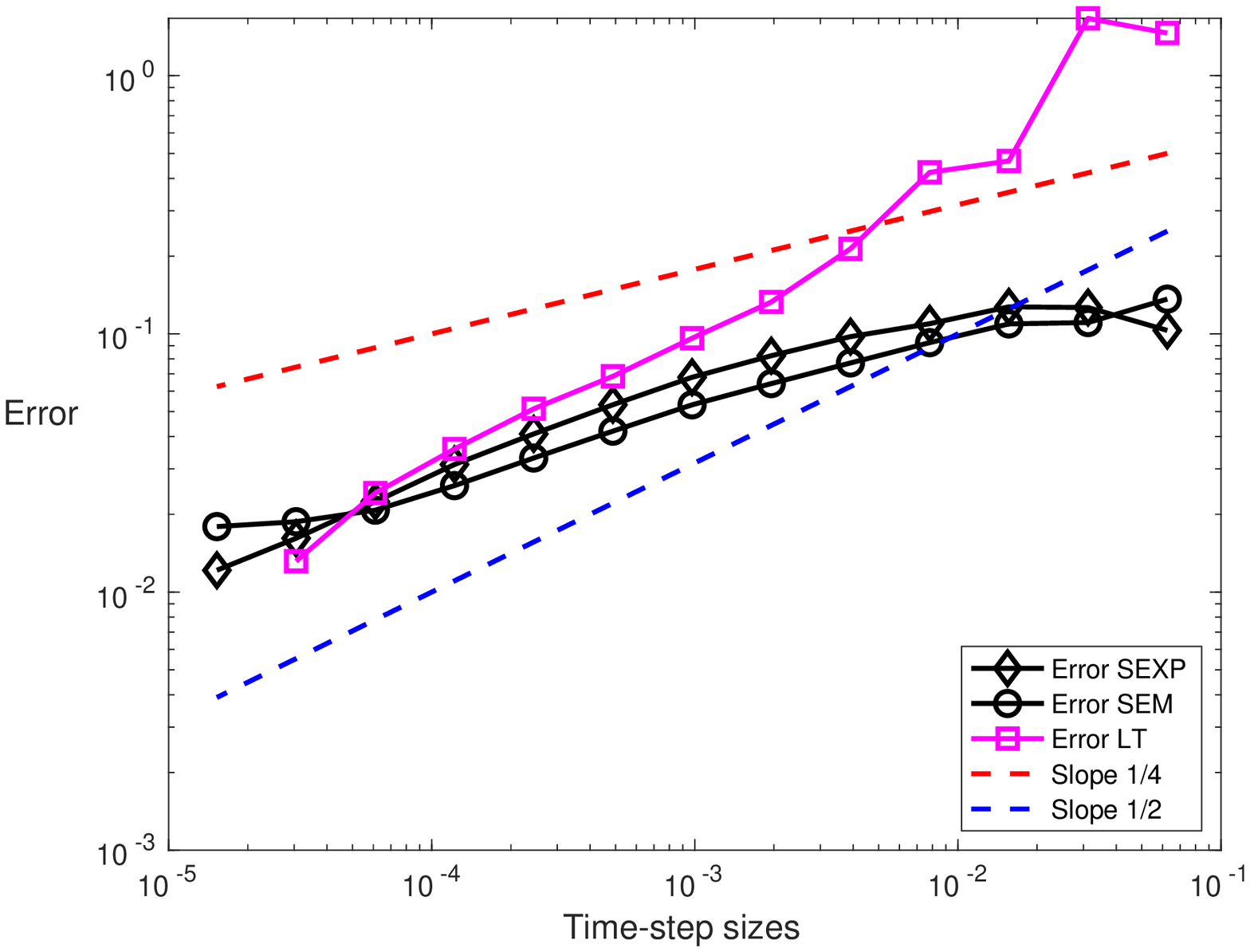}
  \caption{$g(v)=v$}
\end{subfigure}%
\begin{subfigure}{.35\textwidth}
  \centering
  \includegraphics[width=\textwidth]{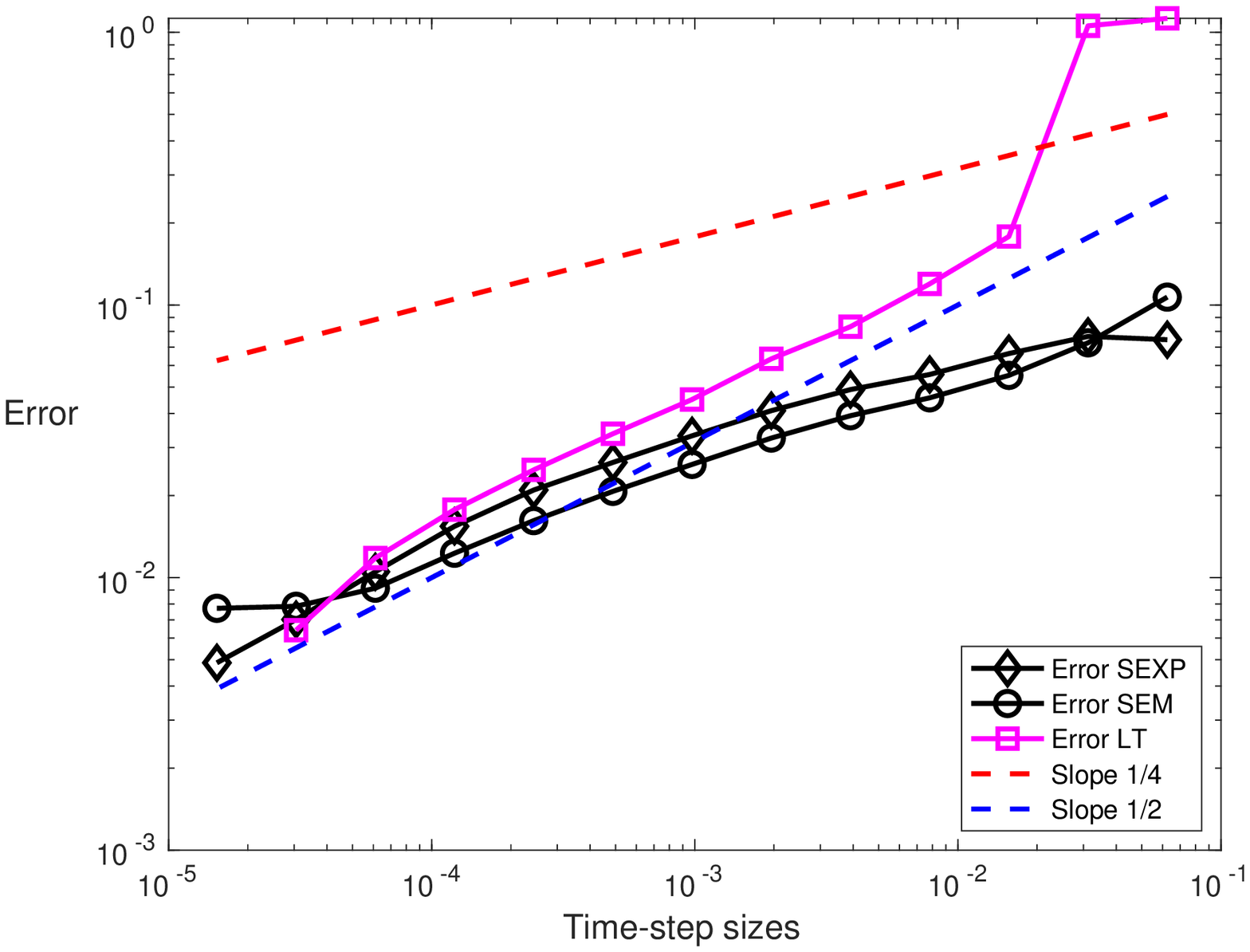}
  \caption{$g(v)=\frac{v}{(1+v^2)}$}
\end{subfigure}%
~

\begin{subfigure}{.35\textwidth}
  \centering
  \includegraphics[width=\textwidth]{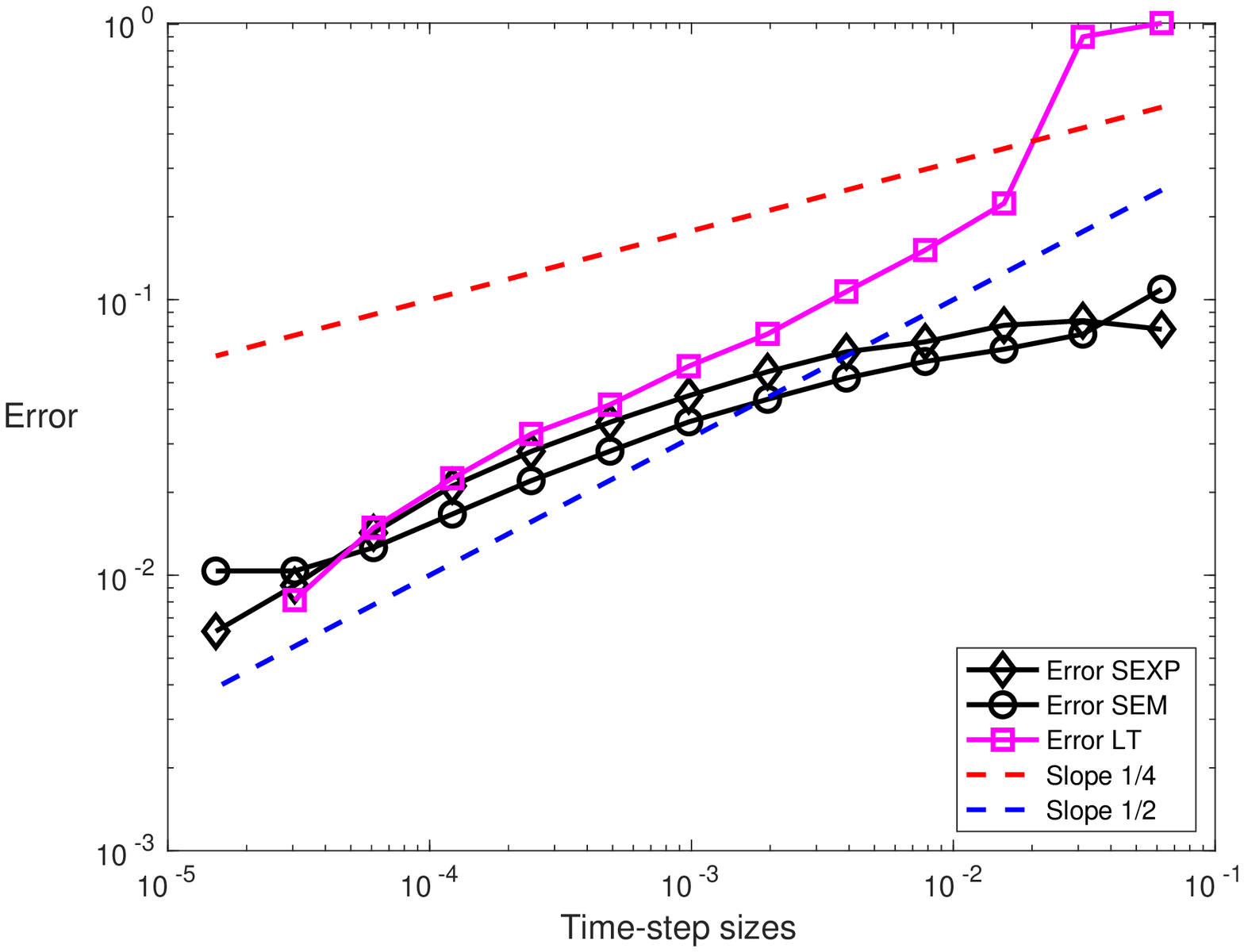}
  \caption{$g(v)=\ln(1+v)$}
\end{subfigure}
\begin{subfigure}{.35\textwidth}
  \centering
  \includegraphics[width=\textwidth]{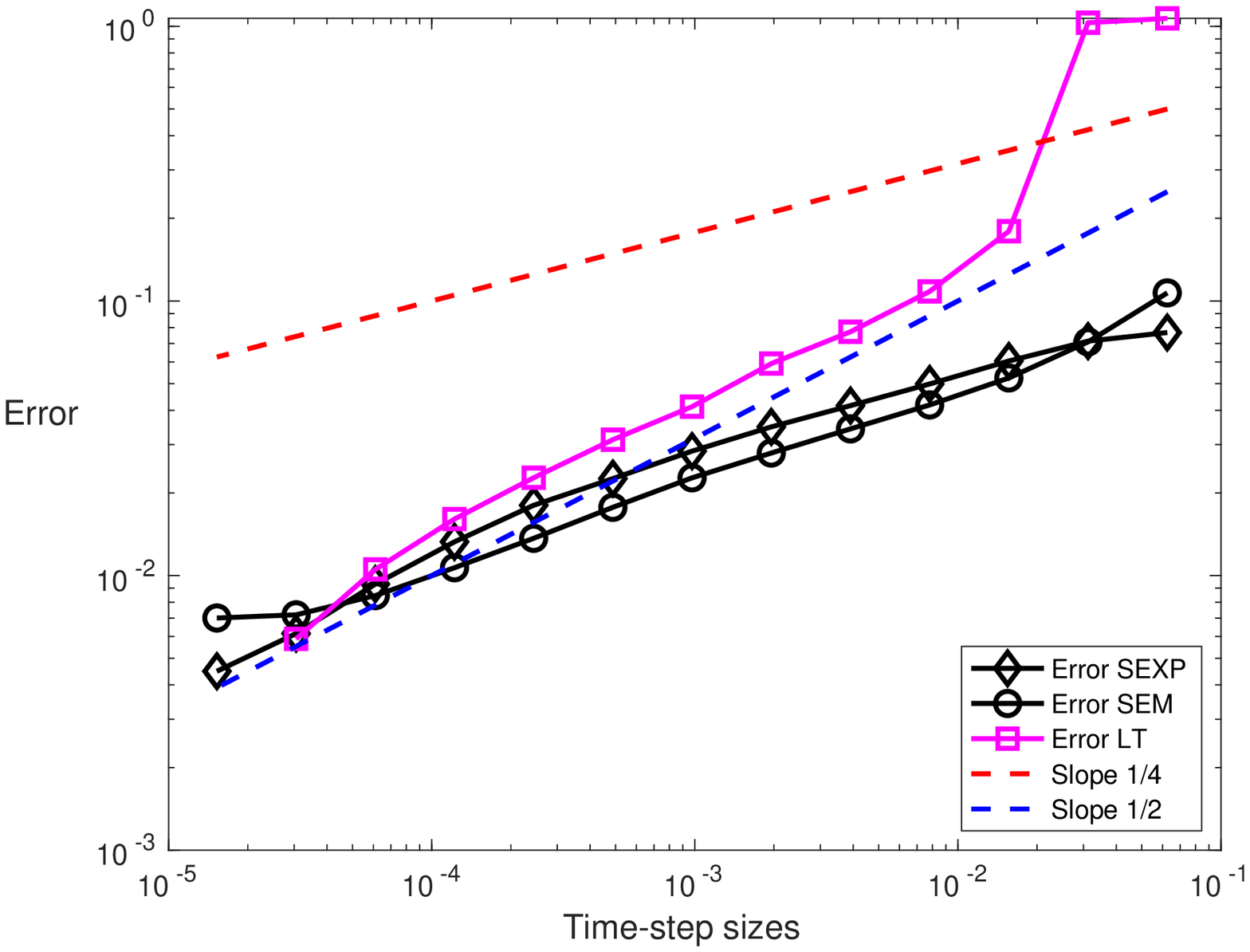}
  \caption{$g(v)=v\exp(-v^2)$}
\end{subfigure}

\caption{Mean-square errors on the time interval $[0,0.5]$ of the splitting scheme (LT), the stochastic exponential Euler integrator (SEXP), and the semi-implicit Euler--Maruyama scheme (SEM). Mesh size $h=2^{-8}$ and average over $200$ samples.
}
\label{fig:strong1}
\end{figure}

To illustrate this, we compute the mean-square errors of the Lie--Trotter splitting scheme when applied to
the finite difference discretization of the stochastic heat equation with different values of the mesh size, namely $h=2^{-4}, 2^{-6}, 2^{-8}, 2^{-10}$.
This is presented only for the two nonlinearities $g(v)=1.5v$ and $g(v)=1.5\frac{v}{1+v^2}$.
We have used $200$ samples to approximate the expectations. The other parameters are the same as in the previous numerical experiments.
The results are presented in Figure~\ref{fig:changeN}. In these experiments we observe upper bounds which are not uniform with respect to $h$, in fact we observe the contribution of the error term $\tau^{\frac12}h^{-\frac12}$ in the mean-square error estimates~\eqref{eq:error}.

\begin{figure}[h]
\begin{subfigure}{.45\textwidth}
  \centering
  \includegraphics[width=\textwidth]{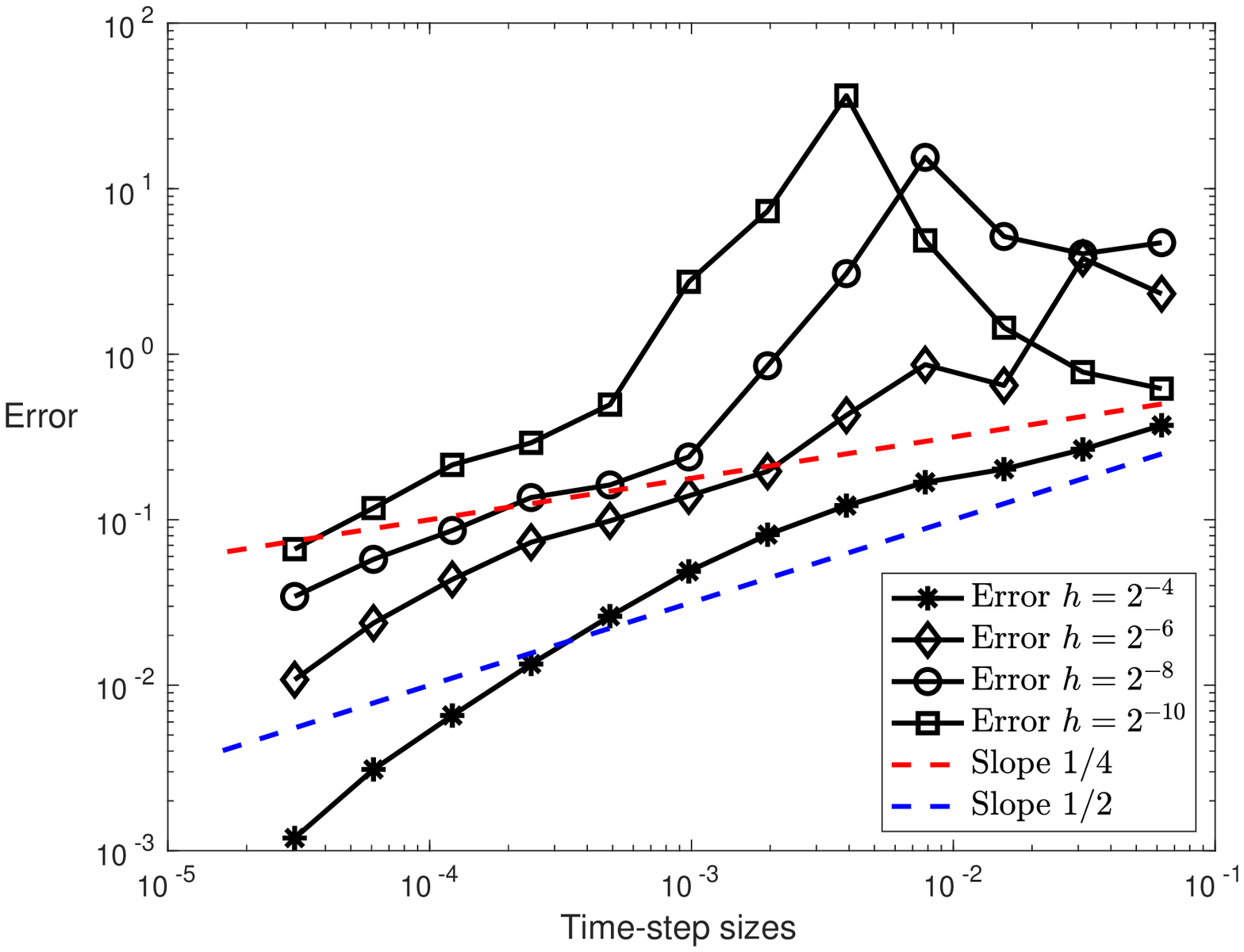}
  \caption{$g(v)=v$}
\end{subfigure}%
\begin{subfigure}{.45\textwidth}
  \centering
  \includegraphics[width=\textwidth]{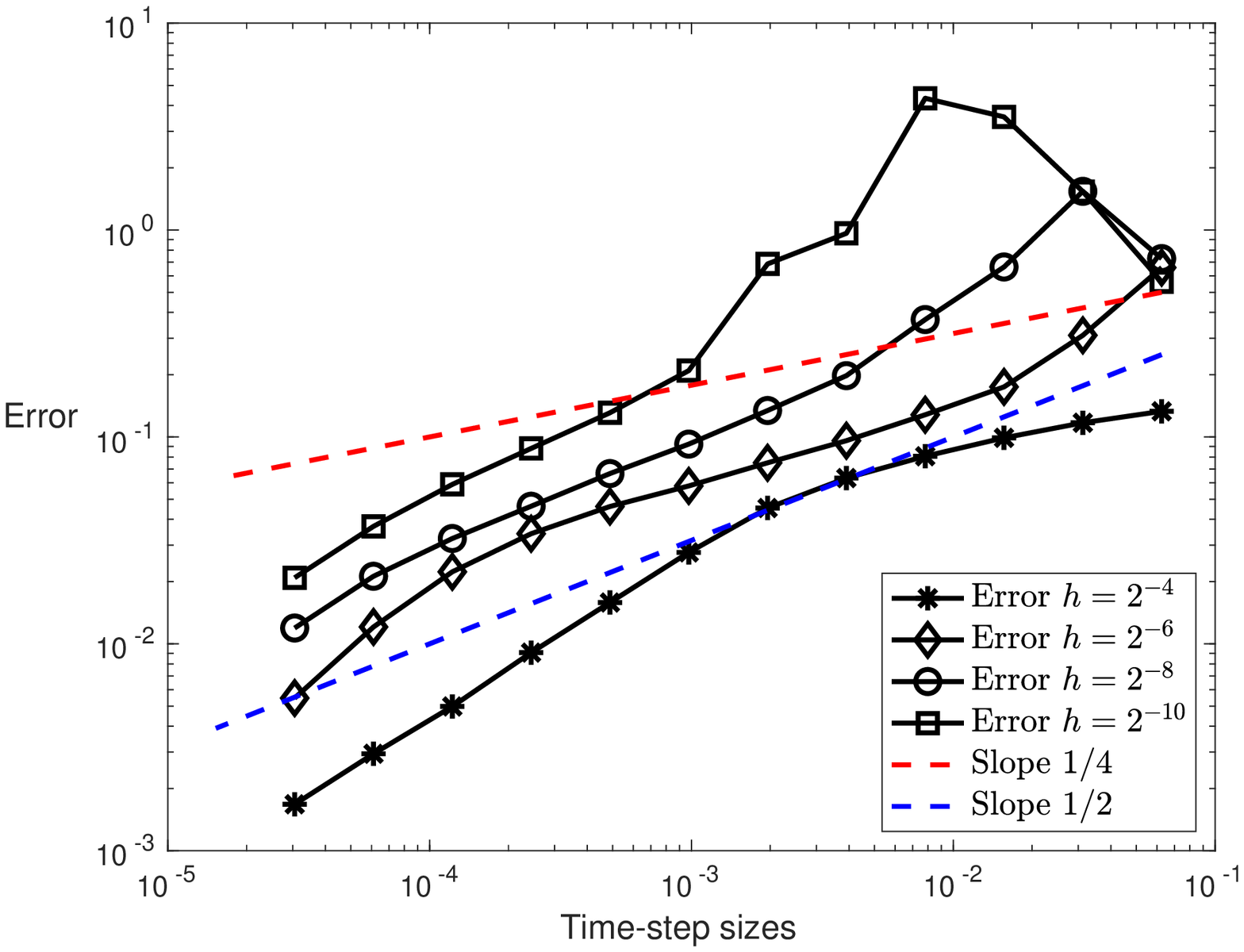}
  \caption{$g(v)=\frac{v}{(1+v^2)}$}
\end{subfigure}%

\caption{Mean-square errors on the time interval $[0,0.5]$ of the splitting scheme for several values of the spatial mesh $h$. Average over $200$ samples.}
\label{fig:changeN}
\end{figure}

In the final numerical experiment, we consider the same parameters as above and the function $g(v)=v^{1.25}$.
Observe that this nonlinearity is not globally Lipschitz continuous and is thus not covered by the the results from Section~\ref{sec:results}. A convergence plot for the splitting scheme~\eqref{eq:scheme}
is provided in Figure~\ref{fig:strong2}. As above, we observe a mean-square order of convergence $1/2$, but which should not be uniform with respect to $h$, similarly to what is observed in Figure~\ref{fig:changeN}. To prove such rate of convergence is beyond the scope of this paper and will be the subject of a future work.

\begin{figure}[h]
  \centering
  \includegraphics[width=.45\textwidth]{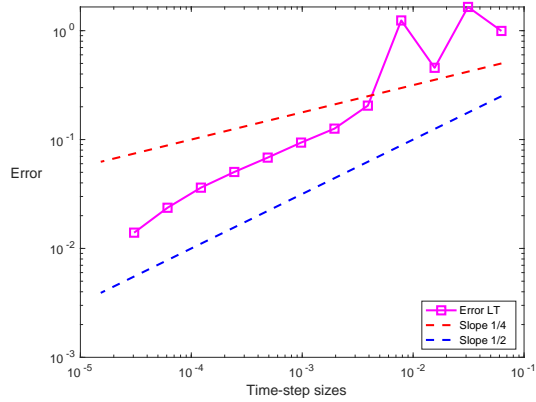}
\caption{Mean-square errors on the time interval $[0,0.5]$ of the splitting scheme (LT) when applied to the stochastic heat equation~\eqref{eq:spde}
with $g(v)=v^{1.25}$. Mesh size $h=2^{-8}$ and average over $200$ samples.}
\label{fig:strong2}
\end{figure}

\section{Proofs of the main results}\label{sec:proofs}

The objective of this section is to provide the proof of the results stated in Section~\ref{sec:results}, namely the moment bounds in Proposition~\ref{propo:moment} and the mean-square error estimates in Theorem~\ref{theo:main} and in Corollary~\ref{cor:errorfull}. We also prove Proposition~\ref{propo:positivity-exact}, which ensures positivity of the exact solution. Preliminary auxiliary tools are given in Sections~\ref{sec:auxiliaryprocess} and~\ref{sec:ineqs}, before proceeding with the detailed proofs.

\subsection{Auxiliary process}\label{sec:auxiliaryprocess}

In this section, for any $M\in\N$ and $N\in\N$, we define an auxiliary stochastic process $\bigl(u^{\LT}(t,x)\bigr)_{t\in[0,T],x\in[0,1]}$ satisfying $u^{\LT}(t_m,x_n)=u_{m,n}^{\LT}$ for all $m\in\{0,\ldots,M\}$ and $n\in\{1,\ldots,N-1\}$. The auxiliary process $u^{\LT}$ is piecewise continuous with respect to the spatial variable $x$, while its temporal evolution on each interval $(t_m,t_{m+1})$ follows a stochastic differential equation similar to~\eqref{eq:scheme-subsystem1}.

Recall that the auxiliary mappings $\kappa^N\colon[0,1]\to\{x_0,\ldots,x_N\}$ and $\ell^M\colon[0,T]\to\{t_0,\ldots,t_M\}$ are defined in Sections~\ref{sec:spatial-discretization} and~\ref{sec:num} respectively.

Let $n\in\{1,\ldots,N-1\}$ and $m\in\{0,\ldots,M-1\}$, then for all $t\in[t_m,t_{m+1}]$ set
\begin{equation}\label{eq:schemeaux1}
u_{m,n}^{\LT}(t)=\sum_{k=1}^{N-1}G_{nk}^{N}(\ell^{M}(t)-t_m)v_{m,k}^{M,N,1}(t),
\end{equation}
where $v_{m,n}^{M,N,1}(t)=\exp\left(\sqrt{N}f(u_{m,n}^{\LT})\bigl(W_n^N(t)-W_n^N(t_m)\bigr)-\frac{Nf(u_{m,n}^{\LT})^2(t-t_m)}{2}\right)u_{m,n}^{\LT}$
is the explicit expression~\eqref{eq:scheme-subsystem1-solution} of the solution at time $t\in[t_m,t_{m+1}]$ of the auxiliary stochastic subsystem~\eqref{eq:scheme-subsystem1} used in the construction of the splitting integrator. Observe that $u_{m,n}^{\LT}(t)=v_{m,n}^{M,N,1}(t)$ for all $t\in[t_m,t_{m+1})$, and, in particular, that $u_{m,n}^{\LT}(t_m)=u_{m,n}^{\LT}$. Moreover, by the construction of the splitting scheme, see~\eqref{eq:scheme}, it holds that $u_{m,n}^{\LT}(t_{m+1})=u_{m+1,n}^{\LT}$.

As a result, for any $n\in\{1,\ldots,N-1\}$, the mapping $u_n^{\LT}\colon t\in [0,T]\mapsto u_n^{\LT}(t)$ defined such that $u_n^{\LT}(t)=u_{m,n}^{\LT}(t)$ for $t\in[t_m,t_{m+1}]$ is well-defined. It is continuous on each interval $[t_m,t_{m+1})$, and one has $u_{n}^{\LT}(t_m)=u_{m,n}^{\LT}$ for all $m\in\{0,\ldots,M\}$.

We claim that the following identity holds: for all $M\in\N$ and $N\in\N$, for all $n\in\{1,\ldots,N-1\}$ and $t\in[0,T]$, one has
\begin{equation}\label{eq:schemeaux2}
u_{n}^{\LT}(t)=\sum_{k=1}^{N-1}G_{nk}^{N}(\ell^M(t))u_{k,0}^{\LT}+\sqrt{N}\int_{0}^{t}\sum_{k=1}^{N-1}G_{nk}^{N}(\ell^M(t)-\ell^M(s))u_k^{\LT}(s)f(u_k^{\LT}(\ell^{M}(s))\,\text dW_k^N(s).
\end{equation}
The proof is based on a straightforward recursion argument.

Recall from Section~\ref{sec:spatial-discretization} that one has the identities $NG_{nk}^{N}(t)=G^{N}(t,x_n,x_k)$ and $\sqrt{N}\text dW_n^N(t)=N\bigl(W(t,x_{n+1})-W(t,x_n)\bigr)$.
We are now in position to provide the definition of the auxiliary process $u^{\LT}$: for $t\in[0,T]$ and $x\in[0,1]$, define
\begin{equation}\label{eq:schemeaux3}
\begin{aligned}
u^{\LT}(t,x)&=\int_0^1 G^N(t,x,y)u_0(\kappa^N(y))\,\text dy\\
&\quad+\int_0^t \int_0^1 G^N(\ell^M(t)-\ell^M(s),x,y)u^{\LT}(s,\kappa^N(y))f(u^{\LT}(\ell^M(s),\kappa^N(y))\,\text dW(s,y).
\end{aligned}
\end{equation}
In the identity~\eqref{eq:schemeaux3} above, it is worth recalling that $x\in[0,1]\mapsto G^N(t,x,y)$ is a piecewise linear mapping, whereas $y\in[0,1]\mapsto G^N(t,x,y)$ is a piecewise constant mapping, with $G^N(t,x_n,x_k)=NG_{nk}^N(t)$ for all $1\le n,k\le N-1$ and $t\in[0,T]$.

Combining~\eqref{eq:schemeaux2} and~\eqref{eq:schemeaux3}, one obtains the identity $u^{\LT}(t,x_n)=u_n^{\LT}(t)$ for all $t\in[0,T]$ and $n\in\{1,\ldots,N-1\}$, and therefore one obtains the required property $u^{\LT}(t_m,x_n)=u_n^{\LT}(t_m)=u_{m,n}^{\LT}$. Note that, for any $t\in[0,T]$, the mapping $x\in[0,1]\mapsto u^{\LT}(t,x)$ is piecewise linear, more precisely it is linear on each subinterval $[x_n,x_{n+1}]$.

\subsection{Auxiliary inequalities}\label{sec:ineqs}

In this subsection we state several inequalities used in the convergence analysis of the splitting scheme.
\begin{itemize}
\item For any continuous function $v\colon[0,1]\to\R$, one has (see for instance \cite[Eq.~(3.5)]{MR1644183})
\begin{equation}\label{eq:aux1}
\underset{N\in\N}\sup~\underset{t\ge 0}\sup~\underset{x\in[0,1]}\sup~\Big|\int_0^1 G^N(t,x,y)v(\kappa^N(y))\,\text dy\Big|\le \underset{x\in[0,1]}\sup~|v(x)|.
\end{equation}
\item For all $T\in(0,\infty)$, there exists $C_T\in(0,\infty)$ such that for all $t\in(0,\infty)$ one has (see for instance \cite[Lemma~2.3]{MR4050540})
\begin{equation}\label{eq:aux2}
\underset{N\in\N}\sup~\underset{x\in[0,1]}\sup~\int_0^1 |G^N(t,x,y)|^2\,\text dy \le \frac{C_T}{\sqrt{t}}.
\end{equation}
\item For all $T\in(0,\infty)$, there exists $C_{T}\in(0,\infty)$ such that for all $t\in(0,T]$ and all $M\in\N$ one has
\begin{equation}\label{eq:aux3}
\underset{N\in\N}\sup~\underset{x\in[0,1]}\sup~\int_0^t\int_0^1\big|G^N(t-s,x,y)-G^N(t-\ell^M(s),x,y)\big|^2\,\text dy\,\text ds \le C_{T}\sqrt{\tau}.
\end{equation}
\end{itemize}
Since we are not aware of a detailed proof of the inequality~\eqref{eq:aux3} in the literature, we provide a proof in Appendix~\ref{sec:app}.
Note that the proof is similar to the proof of \cite[Lemma~2.3]{MR4050540}.

Let us also recall the following discrete Gr\"onwall inequality, see for instance~\cite[Lemma~A.4]{MR3154916}: assume that a sequence $\bigl(a_m\bigr)_{0\le m\le M}$ of nonnegative numbers satisfies the inequality
\[
a_m\le A+C\tau\sum_{k=0}^{m-1}\frac{a_k}{\sqrt{t_m-t_k}},
\]
where we recall that $t_k=k\tau=\frac{kT}{M}$, for some $A,C\in(0,\infty)$. Then, there exists $C_T\in(0,\infty)$, depending only on $C$ and on $T$, such that one has
\begin{equation}\label{eq:gronwall}
\underset{0\le m\le M}\sup~a_m\le C_T A.
\end{equation}

\subsection{Moment bounds}\label{sec:proof-moments}

The objective of this section is to prove Proposition~\ref{propo:moment}. Recall that this requires to impose the condition $\tau\le \gamma h$
where we recall that $\tau=T/M$, $h=1/N$ and where $\gamma\in(0,\infty)$ is an arbitrary parameter.

\begin{proof}[Proof of Proposition~\ref{propo:moment}]
Using the definition~\eqref{eq:schemeaux3} of the auxiliary process $u^{\LT}$, for all $m\in\{1,\ldots,M\}$ and $n\in\{1,\ldots,N-1\}$, one has
\begin{align*}
u_{m,n}^{\LT}&=u^{\LT}(t_m,x_n)\\
&=\int_0^1 G^N(t_m,x_n,y)u_0(\kappa^N(y))\,\text dy\\
&\quad+\int_0^t \int_0^1 G^N(t_m-\ell^M(s),x_n,y)u^{\LT}(s,\kappa^N(y))f(u^{\LT}(\ell^M(s),\kappa^N(y))\,\text dW(s,y).
\end{align*}
Using It\^o's isometry formula, one obtains
\begin{align*}
\E[|u_{m,n}^{\LT}|^2]&=\E[\big|\int_0^1 G^N(t_m,x_n,y)u_0(\kappa^N(y))\,\text dy\big|^2]\\
&\quad+\int_0^t\int_0^1 |G^N(t_m-\ell^M(s),x_n,y)|^2\E[|u^{\LT}(s,\kappa^N(y))|^2|f(u^{\LT}(\ell^M(s),\kappa^N(y))|^2]\,\text dy\,\text ds.
\end{align*}
On the one hand, using the auxiliary inequality~\eqref{eq:aux1} and Assumption~\ref{ass:u0}, one obtains
\[
\E[\big|\int_0^1 G^N(t_m,x_n,y)u_0(\kappa^N(y))\,\text dy\big|^2]\le \|u_0\|_\infty^2.
\]
On the other hand, recall that Assumption~\ref{ass:g} implies that $f$ is bounded by $\Lg$. In addition, for all $k\in\{0,\ldots,m-1\}$ and all $s\in[t_k,t_{k+1})$, one has
\[
\E[|u^{\LT}(s,\kappa^N(y))|^2]=\E[|v_{k,n}^{M,N,1}(s)|^2]
\]
where $n \in \{1, \ldots, N-1 \}$ is such that $\kappa^N(y)=x_n$ and $\bigl(v_{k,n}^{M,N,1}(s)\bigr)_{s\in[t_k,t_{k+1}]}$ is the solution of the auxiliary stochastic subsystem~\eqref{eq:scheme-subsystem1}. Using the expression~\eqref{eq:scheme-subsystem1-solution} for the solution of~\eqref{eq:scheme-subsystem1} and the tower property of conditional expectation, one obtains the upper bound
\[
\E[|v_{k,n}^{M,N,1}(s)|^2]=\E[e^{\frac{Nf(u_{k,n}^{\LT})^2(s-t_k)}{2}}|u_{k,n}^{\LT}|^2]\le e^{\frac{N\tau \Lg^2}{2}}\E[|u_{k,n}^{\LT}|^2]\le e^{\frac{\Lg^2 \gamma}{2}}\E[|u_{k,n}^{\LT}|^2],
\]
using the boundedness of $f$ and the condition $N\tau\le \gamma$.

Using the auxiliary inequality~\eqref{eq:aux2}, gathering the upper bounds above yields the following inequality: for all $m\in\{1,\ldots,M\}$ one has
\[
\underset{1\le n\le N-1}\sup~\E[|u_{m,n}^{\LT}|^2]\le \|u_0\|_\infty^2+C_{\gamma,T}\tau \sum_{k=0}^{m-1}\frac{1}{\sqrt{t_m-t_k}}\underset{1\le n\le N-1}\sup~\E[|u_{k,n}^{\LT}|^2].
\]
Using the discrete Grönwall inequality~\eqref{eq:gronwall} then gives
\begin{equation}\label{eq:proofmoment1}
\underset{0\le m\le M}\sup~\underset{1\le n\le N-1}\sup~\E[|u_{m,n}^{\LT}|^2]\le C_{\gamma,T}\|u_0\|_\infty^2,
\end{equation}
where $C_{\gamma,T}\in(0,\infty)$ is independent of $M$, $N$ and $\|u_0\|_\infty^2$. This shows moment bounds of the numerical solution at the grid.
It remains to extend this moment bound for $u^{\LT}(t,x_n)$ when $t$ is no longer assumed to be a grid point $t_m$.

For all $t\in[0,T)$ and $n\in\{0,\ldots,N-1\}$, let $m \in \{0, \ldots, M-1 \}$ be such that $t_m=\ell^M(t)$, using the same arguments as above one has
\[
\E[|u^{\LT}(t,x_n)|^2]=\E[|v_{m,n}^{M,N,1}(t)|^2]\le e^{\frac{\Lg^2 \gamma}{2}}\E[|u_{k,n}^{\LT}|^2]\le C_{\gamma,T}\|u_0\|_\infty^2,
\]
where the inequality~\eqref{eq:proofmoment1} is used in the last step. As a consequence, one has
\begin{equation}\label{eq:proofmoment2}
\underset{t\in[0,T]}\sup~\underset{1\le n\le N-1}\sup~\E[|u^{\LT}(t,x_n)|^2]\le C_{\gamma,T}\|u_0\|_\infty^2.
\end{equation}
Finally, since $x\mapsto u^{\LT}(t,x)$ is linear on each subinterval $[x_n,x_{n+1}]$, one obtains
\begin{equation}\label{eq:proofmoment3}
\underset{t\in[0,T]}\sup~\underset{x\in[0,1]}\sup~\E[|u^{\LT}(t,x)|^2]\le \underset{t\in[0,T]}\sup~\underset{1\le n\le N-1}\sup~\E[|u^{\LT}(t,x_n)|^2]\le C_{\gamma,T}\|u_0\|_\infty^2.
\end{equation}
The proof of Proposition~\ref{propo:moment} is thus completed.
\end{proof}

A straightforward consequence of Proposition~\ref{propo:moment} is the following result.
\begin{lemma}\label{lem:regul-uLT}
Let Assumption~\ref{ass:u0} and Assumption~\ref{ass:g} be satisfied. Let $\left(u^{\LT}(t,x)\right)_{t\in[0,T], x\in[0,1]}$ be given by the mild formula~\eqref{eq:schemeaux3}.

For all $\gamma\in(0,\infty)$ and all $T\in(0,\infty)$, there exists $C_{\gamma,T}\in(0,\infty)$ such that for all $\tau=T/M$ and $h=1/N$ satisfying the condition $\tau \le \gamma h$, for all $m\in\{0,\ldots,M-1\}$ and all $t\in[t_m,t_{m+1})$, one has
\begin{equation}\label{eq:regul-uLT}
\underset{1\le n\le N-1}\sup~\bigl(\E[|u^{\LT}(t,x_n)-u^{\LT}(t_m,x_n)|^2]\bigr)^{\frac12}\le C_{\gamma,T}(1+\|u_0\|_\infty)\left(\frac\tau h\right)^{\frac12}.
\end{equation}
\end{lemma}

\begin{proof}[Proof of Lemma~\ref{lem:regul-uLT}]
Let $n\in\{1,\ldots,N-1\}$ and $m\in\{0,\ldots,M-1\}$, then for all $t\in[t_m,t_{m+1})$ one has
\begin{align*}
u^{\LT}(t,x_n)-u^{\LT}(t_m,x_n)&=v_{m,n}^{M,N,1}(t)-v_{m,n}^{M,N,1}(t_m)\\
&=\sqrt{N}\int_{t_m}^{t}v_{m,n}^{M,N,1}(s)f(u_{m,n}^{\LT})\,\text dW_n^N(s)\\
&=\sqrt{N}\int_{t_m}^{t}u^{\LT}(s,x_n)f(u_{m,n}^{\LT})\,\text dW_n^N(s),
\end{align*}
where we recall that the auxiliary process $\bigl(v_{m,n}^{M,N,1}(t)\bigr)_{t_m\le t\le t_{m+1}}$ is defined by the auxiliary subsystem~\eqref{eq:scheme-subsystem1} which gives the first step of the splitting procedure, see Section~\ref{sec:schemeconstruction}.

Since the mapping $f$ is bounded, using It\^o's isometry formula, the condition $\tau N\le \gamma$ and the moment bounds~\eqref{eq:momentscheme} from Proposition~\ref{propo:moment}, one obtains
\[
\E[|u^{\LT}(t,x_n)-u^{\LT}(t_m,x_n)|^2]\le \Lg^2 N\tau \E[|u_{m,n}^{\LT}|^2]\le \Lg^{2} C_{\gamma,T}\bigl(1+\|u_0\|_\infty^2\bigr)\tau h^{-1}.
\]
The proof of Lemma~\ref{lem:regul-uLT} is thus completed.
\end{proof}

\subsection{Convergence analysis}\label{sec:proof-convergence}

This section is devoted to the proof of the mean-square convergence of the splitting scheme given in Theorem~\ref{theo:main}.

\begin{proof}[Proof of Theorem~\ref{theo:main}]
Recall that $u_{m,n}^{\LT}=u^{\LT}(t_m,x_n)$ for all $n\in\{1,\ldots,N-1\}$ and $m\in\{0,\ldots,M\}$, where $\bigl(u^{\LT}(t,x)\bigr)_{t\in[0,1],x\in[0,1]}$ is the process defined by~\eqref{eq:schemeaux1}.

For all $n\in\{1,\ldots,N-1\}$ and $m\in\{1,\ldots,M\}$, let us define
\[
E_{m,n}=u^N(t_m,x_n)-u_{m,n}^{\LT}\quad\text{and}\qquad E_m=\underset{1\le n\le N-1}\sup~\E[|E_{m,n}|^2].
\]

Using the expression~\eqref{eq:uN} for $u^N(t,x)$ and the expression~\eqref{eq:schemeaux3} for $u^{\LT}(t,x)$, one obtains the following decomposition of the error: for all $n\in\{1,\ldots,N-1\}$ and $m\in\{1,\ldots,M\}$, one has
\begin{align*}
E_{m,n}&=u^N(t_m,x_n)-u^{\LT}(t_m,x_n)\\
&=\int_0^{t_m}\int_0^1 G^N(t_m-s,x,y)g(u^N(s,\kappa^N(y)))\,\text dW(s,y)\\
&-\int_0^{t_m} \int_0^1 G^N(t_m-\ell^M(s),x,y)u^{\LT}(s,\kappa^N(y))f(u^{\LT}(\ell^M(s),\kappa^N(y))\,\text dW(s,y)\\
&=E_{m,n}^{(1)}+E_{m,n}^{(2)},
\end{align*}
where we set
\begin{align*}
E_{m,n}^{(1)}&=\int_0^{t_m}\int_0^1 G^N(t_m-s,x,y)\bigl[g(u^N(s,\kappa^N(y)))-u^{\LT}(s,\kappa^N(y))f(u^{\LT}(\ell^M(s),\kappa^N(y))\bigr]\,\text dW(s,y),\\
E_{m,n}^{(2)}&=\int_0^{t_m} \int_0^1 \bigl[G^{N}(t_{m}-s,x,y)-G^N(t_m-\ell^M(s),x,y)\bigr]u^{\LT}(s,\kappa^N(y))f(u^{\LT}(\ell^M(s),\kappa^N(y))\,\text dW(s,y).
\end{align*}

Let us first deal with the error term $E_{m,n}^{(1)}$. Recall that $g(u)=uf(u)$, therefore one has the decomposition $E_{m,n}^{(1)}=E_{m,n}^{(1,1)}+E_{m,n}^{(1,2)}+E_{m,n}^{(1,3)}$, where
\begin{align*}
E_{m,n}^{(1,1)}&=\int_0^{t_m}\int_0^1 G^N(t_m-s,x,y)\bigl[g(u^N(s,\kappa^N(y)))-g(u^N(\ell^{M}(s),\kappa^N(y)))\bigr]\, \text dW(s,y)\\
E_{m,n}^{(1,2)}&=\int_0^{t_m}\int_0^1 G^N(t_m-s,x,y)\bigl[g(u^N(\ell^{M}(s),\kappa^N(y)))-g(u^{\LT}(\ell^{M}(s),\kappa^N(y)))\bigr]\,\text dW(s,y)\\
E_{m,n}^{(1,3)}&=\int_0^{t_m}\int_0^1 G^N(t_m-s,x,y)\bigl[u^{\LT}(\ell^{M}(s),\kappa^N(y))-u^{\LT}(s,\kappa^N(y))\bigr]f(u^{\LT}(\ell^{M}(s),\kappa^N(y)))\,\text dW(s,y).
\end{align*}
Using It\^o's isometry formula, the global Lipschitz continuity assumption on $g$, one obtains
\begin{align*}
\E[|E_{m,n}^{(1,1)}|^2]&\le \Lg^2\int_0^{t_m}\int_0^1 G^N(t_m-s,x,y)^2\E[|u^N(s,\kappa^N(y))-u^N(\ell^{M}(s),\kappa^N(y))|^2]\,\text dy\,\text ds\\
&\le C_T(u_0)\Lg^2\sqrt{\tau} \int_0^{t_m}\int_0^1 G^N(t_m-s,x,y)^2 \,\text dy\,\text ds\\
&\le C_T(u_0)\sqrt{\tau},
\end{align*}
where we have used the temporal regularity estimate~\eqref{eq:regul-uN} for $u^N$ and the auxiliary inequality~\eqref{eq:aux2}.

Similarly, using It\^o's isometry formula, the global Lipschitz continuity assumption on $g$, one obtains
\begin{align*}
\E[|E_{m,n}^{(1,2)}|^2]&\le \Lg^2\int_0^{t_m}\int_0^1 G^N(t_m-s,x,y)^2\E[|u^N(\ell^{M}(s),\kappa^N(y))-u^{\LT}(\ell^{M}(s),\kappa^N(y))|^2]\,\text dy\,\text ds\\
&\le C\sum_{k=0}^{m-1}E_{k}\int_{t_k}^{t_{k+1}}\int_0^1 G^N(t_m-s,x,y)^2\,\text dy\,\text ds.
\end{align*}
Using the inequality~\eqref{eq:aux2}, for all $k\in\{0,\ldots,m-1\}$, one has
\begin{align*}
\int_{t_k}^{t_{k+1}}\int_0^1 G^N(t_m-s,x,y)^2\,\text dy\,\text ds&\le \int_{t_k}^{t_{k+1}}\frac{C_T}{\sqrt{t_m-s}}\,\text ds\\
&= 2C_T\bigl(\sqrt{t_m-t_k}-\sqrt{t_m-t_{k+1}}\bigr)\\
&= 2C_T\sqrt{t_m-t_k}\Bigl(1-\sqrt{1-\frac{\tau}{t_m-t_k}}\Bigr)\\
&\le \frac{2C_T\tau}{\sqrt{t_m-t_{k}}},
\end{align*}
where we have used the inequality $1-\sqrt{1-z}\le z$ for all $z\in[0,1]$ in the last step. Therefore one has
\[
\E[|E_{m,n}^{(1,2)}|^2]\le C_T\tau\sum_{k=0}^{m-1}\frac{E_{k}}{\sqrt{t_m-t_k}}.
\]
Finally, for the third term, using It\^o's isometry formula and the boundedness of $f$, one obtains
\begin{align*}
\E[|E_{m,n}^{(1,3)}|^2]&\le L_{g}^{2} \int_0^{t_m}\int_0^1 G^N(t_m-s,x,y)^2\E[|u^{\LT}(\ell^{M}(s),\kappa^N(y))-u^{\LT}(s,\kappa^N(y))|^2]\,\text dy\,\text ds\\
&\le C_{\gamma,T}(u_0)\tau h^{-1}\int_0^{t_m}\int_0^1 G^N(t_m-s,x,y)^2\,
\text dy\,\text ds\\
&\le C_{\gamma,T}(u_0)\tau h^{-1}
\end{align*}
using the temporal regularity estimate~\eqref{eq:regul-uLT} from Lemma~\ref{lem:regul-uLT} for $u^{\LT}$ and the auxiliary inequality~\eqref{eq:aux2}.

Let us now deal with the error term $E_{m,n}^{(2)}$. Using It\^o's formula, the boundedness of $f$ and the moment bounds~\eqref{eq:momentscheme} from Proposition~\ref{propo:moment}, one obtains
\begin{align*}
\E[|E_{m,n}^{(2)}|^2]&\le \Lg^2\int_0^{t_m} \int_0^1 \big|G^{N}(t-\ell^M(s),x,y)-G^N(t_m-\ell^M(s),x,y)\big|^2\E[|u^{\LT}(s,\kappa^N(y))|^2]\,\text dy\,\text ds\\
&\le C_{\gamma,T}(u_0)\int_0^{t_m} \int_0^1 \big|G^{N}(t-\ell^M(s),x,y)-G^N(t_m-\ell^M(s),x,y)\big|^2\,\text dy\,\text ds\\
&\le C_{\gamma,T}(u_0)\sqrt{\tau},
\end{align*}
owing to the auxiliary inequality~\eqref{eq:aux3} in the last step.

Gathering the estimates, for all $m\in\{1,\ldots,M\}$, one has
\[
E_m\le C_{\gamma,T}(u_0)\left(\sqrt{\tau}+\tau h^{-1}\right)+C_T\tau\sum_{k=0}^{m-1}\frac{E_{k}}{\sqrt{t_m-t_k}}.
\]
Applying the discrete Gr\"onwall inequality~\eqref{eq:gronwall} (see Section~\ref{sec:ineqs}) then yields
\[
\underset{0\le m\le M}\sup~E_m\le C_{\gamma,T}(u_0)\left(\sqrt{\tau}+\frac{\tau}{h}\right).
\]
This gives the error estimate~\eqref{eq:error}. When the condition $\tau\le \gamma h^2$ is satisfied, one has $\tau h^{-1}\le \sqrt{\gamma}\tau^{\frac12}$ and one has the error estimate~\eqref{eq:error2}. This concludes the proof of Theorem~\ref{theo:main}.
\end{proof}

Let us also provide the proof of Corollary~\ref{cor:errorfull}.
\begin{proof}[Proof of Corollary~\ref{cor:errorfull}]
It suffices to combine the error estimate~\eqref{eq:errorspatial} from Proposition~\ref{propo:error-spatial} for the spatial discretization error, and the error estimate~\eqref{eq:error} from Theorem~\ref{theo:main} for the temporal discretization error. One then obtains the error estimate for the splitting scheme
\begin{align*}
\bigl(\E[|u_{m,n}^{\LT}-u(t_m,x_n)|^2]\bigr)^{\frac12}&\le \bigl(\E[|u_{m,n}^{\LT}-u^N(t_m,x_n)|^2]\bigr)^{\frac12}+\bigl(\E[|u^N(t_m,x_n)-u(t_m,x_n)|^2]\bigr)^{\frac12}\\
&\le C_{\gamma,T}(u_0)\tau^{\frac14}+C_{T}(u_0)h^{\frac12}\\
&\le C_{\gamma,T}(u_0)\gamma^{\frac14}h^{\frac12}+C_{T}(u_0)h^{\frac12},
\end{align*}
under the condition $\tau\le \gamma h^2$. This gives the error estimate~\eqref{eq:errorfull} and concludes the proof of Corollary~\ref{cor:errorfull}.
\end{proof}

\subsection{Proof of Proposition~\ref{propo:positivity-exact}}\label{sec:proof-positivity-exact}
We conclude this section with the proof of the positivity property of the exact solution to the stochastic heat equation~\eqref{eq:spde} on a bounded domain.

\begin{proof}[Proof of Proposition~\ref{propo:positivity-exact}]
Owing to Corollary~\ref{cor:errorfull} and to the temporal regularity estimate~\eqref{eq:regulexact} satisfied by the solution $u$ of the SPDE in equation~\eqref{eq:spde}, one obtains the following result (recall that $\tau=T/M$ and $h=1/N$):
there exists $C_{\gamma,T}(u_0)\in(0,\infty)$ such that for all $N\in\N$ and $M\in\N$, such that $M\ge \frac{T N^2}{\gamma}$, for all $t\in[0,T]$ and $x\in[0,1]$, one has
\begin{equation}
\bigl(\E[|u(t,x)-u^{\LT}(\ell^{M}(t),\kappa^N(x))|^2]\bigr)^{\frac12}\le C_{\gamma,T}(u_0)N^{-\frac12}.
\end{equation}
Let $t\in[0,T]$ and $x\in[0,1]$ be fixed, then there exists a sequence $\bigl(N_k)_{k\in\N}$ such that $N_{k}\to\infty$ and $u^{\LT}(\ell^{M_k}(t),\kappa^{N_k}(x))$ converges to $u(t,x)$ almost surely. Since $u^{\LT}(\ell^{M_k}(t),\kappa^{N_k}(x))\ge 0$ almost surely owing to Proposition~\ref{propo:positivity-scheme}, one obtains $u(t,x)\ge 0$ almost surely.
\end{proof}

\section{Generalization to systems}\label{sec:syst}

In this section, we briefly describe how to generalize the construction of the splitting scheme~\eqref{eq:scheme} and the analysis above to stochastic systems of the type
\begin{equation}\label{eq:spde-sys}
\left\lbrace
\begin{aligned}
&\text d u_1(t,x) = \partial_{xx}^2 u_1(t,x)\,\text dt + g_1(u_1(t,x),u_2(t,x))\,\text d W_1(t,x),\\
&\text d u_2(t,x) = \partial_{xx}^2 u_2(t,x)\,\text dt + g_2(u_1(t,x),u_2(t,x))\,\text d W_2(t,x),\\
&u_1(t,0)=u_1(t,1)=0,\quad u_2(t,0)=u_2(t,1)=0,\\
&u_1(0,x) = u_{1,0}(x),\quad u_2(0,x) = u_{2,0}(x),
\end{aligned}
\right.
\end{equation}
for $(t,x) \in [0,T] \times [0,1]$, where $g_1,g_2\colon\R^2\to\R$ are globally Lipschitz continuous mappings, with initial values $u_{1,0},u_{2,0}$ satisfying Assumptions~\ref{ass:u0} and~\ref{ass:pos}. The two evolution equations are driven by space-time white noise. The Wiener sheets $W_1$ and $W_2$ can either be equal or independent. For ease of presentation we only deal with systems of two equations, while considering systems of arbitrary size would also be possible.

In this setting, to obtain solutions which only have nonnegative values, it is necessary to replace Assumption~\ref{ass:g} by the following.

\begin{assumption}\label{ass:gsystem}
The mappings $g_1,g_2\colon\R^2\to\R$ are of class $\mathcal{C}^1$ and globally Lipschitz continuous. In addition, they satisfy $g_1(0,v_2)=0$ and $g_2(v_1,0)=0$
for all $(v_1,v_2)\in\R^2$.
\end{assumption}

One then has the following generalization of Proposition~\ref{propo:positivity-exact}.
\begin{proposition}\label{propo:positivity-exact-sys}
Consider the SPDE system~\eqref{eq:spde-sys}. Let Assumption~\ref{ass:gsystem} be satisfied and assume that the initial values $u_{1,0},u_{2,0}$ satisfy Assumptions~\ref{ass:u0} and~\ref{ass:pos}. Then, for all $t\in(0,\infty)$ and all $x\in[0,1]$, almost surely, one has
\[
u_1(t,x)\ge 0~,\quad u_2(t,x)\ge 0.
\]
\end{proposition}

As in Sections~\ref{sec:setting} and~\ref{sec:scheme}, the mesh size and the time-step sizes are denoted by $h=1/N$ and $\tau=T/M$ respectively, and the space and time grid points are denoted by $x_n=nh$ and $t_m=m\tau$, with $0\le n\le N$ and $0\le m\le M$. In addition, introduce the mappings $f_1,f_2\colon\R^2\to\R$ defined by
\begin{align*}
f_1(v_1,v_2)&=\frac{g_1(v_1,v_2)}{v_1}=\int_{0}^{1}\partial_{v_1}g_1(rv_1,v_2)\,\text dr~,\quad
f_2(v_1,v_2)&=\frac{g_2(v_1,v_2)}{v_2}=\int_{0}^{1}\partial_{v_2}g_2(v_1,rv_2)\,\text dr.
\end{align*}
Owing to Assumption~\ref{ass:gsystem}, the mappings $f_1$ and $f_2$ are bounded and continuous mappings.
Finally, for all $t\ge 0$ and $n\in\{1,\ldots,N-1\}$ define
\[
W_{1,n}^{N}(t)=\sqrt{N}\bigl(W_1(t,x_{n+1})-W_1(t,x_n)\bigr)~,\quad W_{2,n}^{N}(t)=\sqrt{N}\bigl(W_2(t,x_{n+1})-W_2(t,x_n)\bigr)
\]
and define the noise increments
\[
\Delta_{m,n}W_1=W_{1,n}^{N}(t_{m+1})-W_{1,n}^{N}(t_m)~,\quad \Delta_{m,n}W_{2}=W_{2,n}^{N}(t_{m+1})-W_{2,n}^{N}(t_m)
\]
for all $n\in\{1,\ldots,N-1\}$ and $m\in\{0,\ldots,M-1\}$.

Using the finite difference method and the same notation as in Section~\ref{sec:spatial-discretization}, one obtains the spatial semi-discretization scheme for the SPDE system~\eqref{eq:spde-sys} with mesh size $h$ as follows:
\begin{equation}\label{eq:spatialscheme-sys}
\left\lbrace
\begin{aligned}
\text du_1^{N}(t)&=N^2D^Nu_1^{N}(t)\,\text dt+\sqrt{N}g_1(u_1^{N}(t),u_2^{N}(t))\,\text dW_1^{N}(t)\\
\text du_2^{N}(t)&=N^2D^Nu_2^{N}(t)\,\text dt+\sqrt{N}g_2(u_1^{N}(t),u_2^{N}(t))\,\text dW_2^{N}(t).
\end{aligned}
\right.
\end{equation}

We are now in position to state the definition of the fully-discrete scheme based on a Lie--Trotter splitting strategy and inspired by~\eqref{eq:scheme} for the approximation of solutions of~\eqref{eq:spde-sys}: for all $m\in\{0,\ldots,M-1\}$, set
\begin{equation}\label{eq:scheme-sys}
\left\lbrace
\begin{aligned}
u_{1,m+1}^{\LT}&=e^{\tau N^2D^N}\left(\exp\Bigl(\sqrt{N}f_1(u_{1,m,n}^{\LT},u_{2,m,n}^{\LT})\Delta_{m,n}W_1-\frac{Nf_1(u_{1,m,n}^{\LT},u_{2,m,n}^{\LT})^2\tau}{2}\Bigr)u_{1,m,n}^{\LT}\right)_{1\le n\le N-1}\\
u_{2,m+1}^{\LT}&=e^{\tau N^2D^N}\left(\exp\Bigl(\sqrt{N}f_2(u_{1,m,n}^{\LT},u_{2,m,n}^{\LT})\Delta_{m,n}W_2-\frac{Nf_2(u_{1,m,n}^{\LT},u_{2,m,n}^{\LT})^2\tau}{2}\Bigr)u_{2,m,n}^{\LT}\right)_{1\le n\le N-1},
\end{aligned}
\right.
\end{equation}
with initial values $u_{1,0}^{\LT}=\bigl(u_{1,0}(x_n)\bigr)_{1\le n\le N-1}$ and $u_{1,0}^{\LT}=\bigl(u_{2,0}(x_n)\bigr)_{1\le n\le N-1}$.

The scheme~\eqref{eq:scheme-sys} is positivity-preserving in the following sense.
\begin{proposition}\label{propo:positivity-scheme-sys}
Let $M\in\N$ and $N\in\N$ be arbitrary integers and let $T\in(0,\infty)$.
Let Assumption~\ref{ass:gsystem} be satisfied, and assume that the initial values $u_{1,0},u_{2,0}$ satisfy Assumptions~\ref{ass:u0} and~\ref{ass:pos}.
Let the sequence $u_{1,0}^{\LT},\ldots,u_{1,M}^{\LT}$ and $u_{2,0}^{\LT},\ldots,u_{2,M}^{\LT}$ be given by the splitting scheme~\eqref{eq:scheme-sys}, with $h=1/N$ and $\tau=T/M$, with initial values $u_{1,0,n}^{\LT}=u_{1,0}(x_n)\ge 0$ and $u_{1,0,n}^{\LT}=u_{2,0}(x_n)\ge 0$ for all $n\in\{1,\ldots,N\}$. Then, almost surely, one has
\[
u_{1,m,n}^{\LT}\ge 0~,\quad u_{2,m,n}^{\LT}\ge 0,
\]
for all $m\in\{1,\ldots,M\}$ and $n\in\{1,\ldots,N-1\}$.
\end{proposition}

The proof of Proposition~\ref{propo:positivity-scheme-sys} is a straightforward modification of the proof of Proposition~\ref{propo:positivity-scheme}. Moreover, one has the following variant of Proposition~\ref{propo:moment}.
\begin{proposition}\label{propo:moment-sys}
Let Assumption~\ref{ass:gsystem} be satisfied and assume that the initial values $u_{1,0},u_{2,0}$ satisfy Assumptions~\ref{ass:u0} and~\ref{ass:pos}.
Let the sequences $u_{1,0}^{\LT},\ldots,u_{1,M}^{\LT}$ and $u_{2,0}^{\LT},\ldots,u_{2,M}^{\LT}$ be given by the Lie--Trotter splitting scheme~\eqref{eq:scheme-sys}.

For all $\gamma\in(0,\infty)$ and all $T\in(0,\infty)$, there exists $C_{\gamma,T}\in(0,\infty)$ such that for all $\tau=T/M$ and $h=1/N$ satisfying the condition $\tau \le \gamma h$, one has
\begin{equation}\label{eq:momentscheme-sys}
\underset{0\le m\le M}\sup~\underset{1\le n\le N-1}\sup~\E\left[|u_{1,m,n}^{\LT}|^2\right]+\underset{0\le m\le M}\sup~\underset{1\le n\le N-1}\sup~\E\left[|u_{2,m,n}^{\LT}|^2\right]\le C_{\gamma,T}\bigl(1+\|u_{1,0}\|_\infty^2+\|u_{2,0}\|_\infty^2\bigr).
\end{equation}
\end{proposition}

Finally, one has the following generalization of Theorem~\ref{theo:main}.
\begin{theorem}\label{theo:main-sys}
Let Assumption~\ref{ass:gsystem} be satisfied and assume that the initial values $u_{1,0},u_{2,0}$ satisfy Assumptions~\ref{ass:u0} and~\ref{ass:pos}. Let the sequences $u_{1,0}^{\LT},\ldots,u_{1,M}^{\LT}$ and $u_{2,0}^{\LT},\ldots,u_{2,M}^{\LT}$ be given by the Lie--Trotter splitting scheme~\eqref{eq:scheme-sys}, and let $\bigl(u_1^{N}(t)\bigr)_{t\ge 0,0\le n\le N}$ and $\bigl(u_2^{N}(t)\bigr)_{t\ge 0,0\le n\le N}$ be given by the spatial semi-discretization scheme~\eqref{eq:spatialscheme-sys}.

For all $\gamma\in(0,\infty)$ and $T\in(0,\infty)$, there exists $C_{\gamma,T}(u_{1,0},u_{2,0})\in(0,\infty)$ such that for all $\tau=T/M$ and $h=1/N$ satisfying the condition $\tau\le \gamma h$, one has
\begin{equation}\label{eq:error-sys}
\begin{aligned}
\underset{0\le m\le M}\sup~\underset{0\le n\le N}\sup~\left(\E[|u_{1,m,n}^{\LT}-u_{1,n}^{N}(t_m)|^2]\right)^{\frac12}&\le C_{\gamma,T}(u_{1,0},u_{2,0})\left(\tau^{\frac{1}{4}}+\left(\frac\tau h\right)^{\frac12}\right)\\
\underset{0\le m\le M}\sup~\underset{0\le n\le N}\sup~\left(\E[|u_{2,m,n}^{\LT}-u_{2,n}^{N}(t_m)|^2]\right)^{\frac12}&\le C_{\gamma,T}(u_{1,0},u_{2,0})\left(\tau^{\frac{1}{4}}+\left(\frac\tau h\right)^{\frac12}\right).
\end{aligned}
\end{equation}
In addition, for all $\tau=T/M$ and $h=1/N$ satisfying the condition $\tau\le \gamma h^2$, one has
\begin{equation}\label{eq:error2-sys}
\begin{aligned}
\underset{0\le m\le M}\sup~\underset{0\le n\le N}\sup~\left(\E[|u_{1,m,n}^{\LT}-u_{1,n}^{N}(t_m)|^2]\right)^{\frac12}&\le C_{\gamma,T}(u_{1,0},u_{2,0})\tau^{\frac{1}{4}}\\
\underset{0\le m\le M}\sup~\underset{0\le n\le N}\sup~\left(\E[|u_{2,m,n}^{\LT}-u_{2,n}^{N}(t_m)|^2]\right)^{\frac12}&\le C_{\gamma,T}(u_{1,0},u_{2,0})\tau^{\frac{1}{4}}.
\end{aligned}
\end{equation}
\end{theorem}

The proofs of Proposition~\ref{propo:moment-sys} and of Theorem~\ref{theo:main-sys} are omitted since they follow from the same arguments as those of Proposition~\ref{propo:moment} and of Theorem~\ref{theo:main}. Finally, one obtains the following variant of Corollary~\ref{cor:errorfull}

\begin{corollary}\label{cor:errorfull-sys}
Consider the setting and assumptions of Theorem~\ref{theo:main-sys}.
For all $\gamma\in(0,\infty)$ and $T\in(0,\infty)$, there exists $C_{\gamma,T}(u_{1,0},u_{2,0})\in(0,\infty)$ such that for all $\tau=T/M$ and $h=1/N$ satisfying the condition $\tau\le \gamma h^2$, one has
\begin{equation}\label{eq:errorfull-sys}
\begin{aligned}
\underset{0\le m\le M}\sup~\underset{0\le n\le N}\sup~\left(\E[|u_{1,m,n}^{\LT}-u_1(t_m,x_n)|^2]\right)^{\frac12}&\le C_{\gamma,T}(u_{1,0},u_{2,0})h^{\frac{1}{2}}\\
\underset{0\le m\le M}\sup~\underset{0\le n\le N}\sup~\left(\E[|u_{2,m,n}^{\LT}-u_2(t_m,x_n)|^2]\right)^{\frac12}&\le C_{\gamma,T}(u_{1,0},u_{2,0})h^{\frac{1}{2}}.
\end{aligned}
\end{equation}
\end{corollary}

To conclude this presentation of the positivity-preserving Lie--Trotter splitting scheme~\eqref{eq:scheme-sys} for the approximation of solutions of the SPDE system~\eqref{eq:spde-sys}, we report some numerical experiments.

The first numerical experiment illustrates the positivity-preserving property of the Lie--Trotter splitting scheme (LT)
when applied to the system of SPDEs~\eqref{eq:spde-sys} driven by two independent noise.
The initial values are taken to be $u_{1,0}=u_{2,0}=\sin(\pi x)$, the final time is $T=5$ and the multiplicative terms
are $g_1(v_1,v_2)=7\sin(v_1)\cos(v_2)$ and $g_2(v_1,v_2)=7\cos(v_1)\sin(v_2)$.
The discretization parameters are $\tau=2^{-2}$ and $h=2^{-8}$.
The proportion of samples containing only positive values out of $500$ simulated samples for all considered time integrators
are presented in Table~\ref{tablesys}.

\begin{table}[h]
\begin{center}
\begin{tabular}{| c c c c|}
 \hline
 LT (first,second) & SEXP (first,second) & SEM (first,second) & EM (first,second)\\
 \specialrule{.1em}{.05em}{.05em}
 $500/500, 500/500$ & $500/500, 499/500$ & $498/500, 496/500$ & $0/500, 0/500$ \\
 \hline
\end{tabular}
\end{center}
\caption{Proportion of samples containing only positive values out of $500$ simulated sample paths for
the Lie--Trotter splitting scheme (LT), the stochastic exponential Euler integrator (SEXP),
the semi-implicit Euler--Maruyama scheme (SEM), and the Euler--Maruyama scheme (EM).
First and second component. The multiplicative terms are $g_1(v_1,v_2)=7\sin(v_1)\cos(v_2)$ and $g_2(v_1,v_2)=7\cos(v_1)\sin(v_2)$.
The discretization parameters are $\tau=2^{-2}$ and $h=2^{-8}$.}\label{tablesys}
\end{table}

The second numerical experiment illustrates the mean-square convergence of the Lie--Trotter splitting scheme when applied to systems of nonlinear SHEs.
Figure~\ref{fig:sys} presents, in a loglog plot, the mean-square errors measured at the space-time grid for the time interval $[0,0.5]$.
The discretization parameters are $h=2^{-8}$ and $\tau=2^{-4},2^{-5},\ldots,2^{-16}$ (the last one being used for the reference solution).
We have used $200$ samples to approximate the expected values. The expected mean-square orders of convergence is observed in this figure.

\begin{figure}[h]
  \centering
  \includegraphics[width=.45\textwidth]{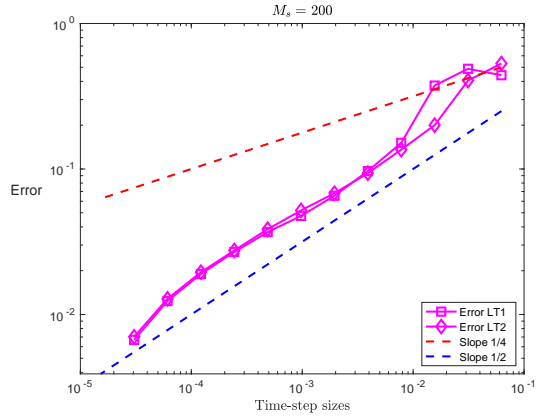}
\caption{Mean-square errors of the Lie--Trotter splitting scheme (first component denoted by LT1, second by LT2)
when applied to the system of stochastic heat equations
with multiplicative terms $g_1(v_1,v_2)=\sin(v_1)\cos(v_2)$ and $g_2(v_1,v_2)=\cos(v_1)\sin(v_2)$.
Mesh size $h=2^{-8}$ and average over $200$ samples.}
\label{fig:sys}
\end{figure}

\begin{appendix}
\section{Proof of auxiliary inequalities}\label{sec:app}

\begin{proof}[Proof of the auxiliary inequality~\eqref{eq:aux3}]

Let us recall some notation. For all $N\in\N$, all $t\ge 0$ and $x,y\in[0,1]$, one has
\[
G^N(t,x,y)=\sum_{j=1}^{N-1}e^{-\lambda_j^N t}\varphi_j^{N}(x)\varphi_j(\kappa^N(y)),
\]
where $\lambda_j^N=4N^2\sin\bigl(\frac{j\pi}{2N}\bigr)^2$, $\varphi_j(\cdot)=\sqrt{2}\sin(j\pi\cdot)$ and $\varphi_j^{N}$ is the linear interpolation of $\varphi_j$ at the space grid points $x_n=nh$ for $n = 1, \ldots, N-1$.

Using the orthogonality property
\[
\int_0^1 \varphi_j(\kappa^N(y))\varphi_k(\kappa^N(y))\,\text dy=\delta_{jk},
\]
one obtains
\begin{align*}
\int_0^{t}\int_0^1\big|G^N(t-s,x,y)&-G^N(t-\ell^M(s),x,y)\big|^2\,\text dy\,\text ds\\
&=\int_{0}^{t}\int_0^1\big|\sum_{j=1}^{N-1}\bigl(e^{-\lambda_j^N(t-s)}-e^{-\lambda_j^N(t-\ell^M(s))}\bigr)\varphi_j^{N}(x)\varphi_j(\kappa^N(y))\big|^2\,\text dy\,\text ds\\
&=\int_{0}^{t}\sum_{j=1}^{N-1}\bigl(e^{-\lambda_j^N(t-s)}-e^{-\lambda_j^N(t-\ell^M(s))}\bigr)^2\varphi_j^{N}(x)^2\,\text ds\\
&\le 2\int_{0}^{t}\sum_{j=1}^{N-1}\bigl(e^{-\lambda_j^N(t-s)}-e^{-\lambda_j^N(t-\ell^M(s))}\bigr)^2\,\text ds\\
&\le 2\sum_{j=1}^{N-1}\int_{0}^{t}e^{-2\lambda_j^N(t-s)}\bigl(1-e^{-\lambda_j^N(s-\ell^M(s))}\bigr)^2\,\text ds\\
&\le C\sum_{j=1}^{N-1}\frac{\max(1,\lambda_j^N\tau)^2}{\lambda_j^N}.
\end{align*}
One checks that there exists $c\in(1,\infty)$ such that for all $N\ge 1$ and $j\in\{1,\ldots,N-1\}$ one has
\[
c^{-1}\le \frac{\lambda_j^N}{j^2}\le c.
\]
Let $L\in\N$ be an arbitrary positive integer. Owing to the inequalities above, one obtains
\begin{align*}
\sum_{j=1}^{N-1}\frac{\max(1,\lambda_j^N\tau)^2}{\lambda_j^N}&\le C\sum_{j=1}^{\infty}\frac{\max(1,j^2\tau)^2}{j^2}\\
&\le C\sum_{j=1}^{L}j^2\tau^2+C\sum_{j=L+1}^{\infty}j^{-2}\\
&\le C\tau^2L^3+CL^{-1},
\end{align*}
using standard comparison of series and integrals arguments. Choosing $L=\lfloor \tau^{-\frac12}\rfloor\ge 1$ (where $\lfloor \cdot\rfloor$ denotes the integer part), and recalling that $\tau\in(0,1)$, one obtains
\[
\int_0^{t}\int_0^1\big|G^N(t-s,x,y)-G^N(t-\ell^M(s),x,y)\big|^2\,\text dy\,\text ds\le C\tau^{\frac12}.
\]
The value of $C$ is independent of $N\in\N$, $t\in(0,T]$ and $x\in[0,1]$. The proof of the auxiliary inequality~\eqref{eq:aux3} is thus completed.
\end{proof}

\end{appendix}

\section*{Acknowledgements}
% We appreciate the referees' comments on an earlier version of the paper.
The work of CEB is partially supported by the project SIMALIN (ANR-19-CE40-0016) operated by the French National Research Agency.
The work of DC and JU is partially supported by the Swedish Research Council (VR) (projects nr. $2018-04443$).
The computations were performed on resources provided by the Swedish National Infrastructure
for Computing (SNIC) at HPC2N, Ume{\aa} University and at UPPMAX, Uppsala University.

\bibliographystyle{abbrv}
\bibliography{labib}

\end{document}